\pdfoutput=1

\documentclass[12pt]{article}

\usepackage{amsmath,amssymb,amsthm}

\usepackage[english]{babel}
\usepackage[T1]{fontenc}
\usepackage[utf8]{inputenc}

\usepackage{subcaption}
\usepackage{tikz}
\usepackage{tikz-cd}
\usepackage{calc}
\usetikzlibrary{decorations.markings,calc,shapes,arrows.meta,positioning,shapes.multipart}
\usetikzlibrary{decorations.pathmorphing}

\usepackage[
  backend=biber,
  style=numeric,
  sorting=none,
  maxnames=50,
  giveninits=true
]{biblatex}
\usepackage{hyperref}
\hypersetup{
  colorlinks=true,
  linkcolor=blue,
  citecolor=blue,
  urlcolor=blue
}

\newcommand{\bb}{\mathbb}
\newcommand{\R}{\bb R}

\newcommand{\Z}{\bb Z}

\newcommand{\floor}[1]{\left\lfloor#1\right\rfloor}
\newcommand{\ceil}[1]{\left\lceil#1\right\rceil}

\newcommand{\sm}{\setminus}

\renewcommand{\Im}{\operatorname{Im}}

\renewcommand{\epsilon}{\varepsilon}

\renewcommand{\int}{\operatorname{int}}
\newcommand{\C}{\mathcal{C}}

\newcommand{\inn}{\mathrm{in}}
\newcommand{\out}{\mathrm{out}}

\newcommand{\sd}{\bigtriangleup}

\newcommand{\RB}{\tikz[baseline]{\draw[->] (0.3,0) -- (0,0.3);}}
\newcommand{\BR}{\tikz[baseline]{\draw[->] (0.3,0.3) -- (0.6,0);}}
\newcommand{\RY}{\tikz[baseline]{\draw[->] (0.3,0) -- (0,0);}}
\newcommand{\BY}{\tikz[baseline]{\draw[->] (0,0.3) -- (0,0);}}
\newcommand{\YR}{\tikz[baseline]{\draw[->] (0,0) -- (0.3,0);}}
\newcommand{\YB}{\tikz[baseline]{\draw[->] (0,0) -- (0,0.3);}}
\newcommand{\YRRB}{\tikz[baseline]{\draw[->] (0,0) -- (0.3,0); \draw[->] (0.35,0) -- (0,0.35);}}
\newcommand{\BRRY}{\tikz[baseline]{\draw[->] (0.0,0.3) -- (0.35,0.05);\draw[->] (0.3,0) -- (0,0);}}
\tikzstyle{pointv}=[circle, draw, fill=black, minimum size=2pt, inner sep=1.5pt]

\newtheorem{prop}{Proposition}
\newtheorem{ass}{Assumption}
\newtheorem{obs}[prop]{Observation}
\newtheorem{remark}[prop]{Remark}
\newtheorem{lemma}[prop]{Lemma}

\newtheorem{thm}[prop]{Theorem}

\newtheorem{conj}[prop]{Conjecture}
\newtheorem{claim}[prop]{Claim}

\newtheorem*{lemma*}{Lemma}
\newtheorem*{thm*}{Theorem}

\newenvironment{cpf}
{\begin{trivlist}\item[] {\em Proof of claim.}}
{\hspace*{\stretch{1}} $\diamond$ \end{trivlist}}

\title{The red-blue-yellow matching problem}

\author{
Manuel Aprile, Marco Di Summa\\[1ex]
\small Universit\`a degli Studi di Padova, Dipartimento di Matematica ``Tullio Levi-Civita''\\
\small Via Trieste 63, 35121 Padova, Italy\\[0.5ex]
\small \texttt{manuel.aprile@unipd.it}, \texttt{disumma@math.unipd.it}
}

\date{}

\begin{document}

\maketitle

\begin{abstract}
We consider the red-blue-yellow matching problem: given two natural numbers $k_R$, $k_B$ and a graph $G$ whose edges are colored red, blue or yellow, the goal is to find a matching of $G$ that contains exactly $k_R$ red edges and exactly $k_B$ blue edges, and is of maximum cardinality subject to these constraints. This is a natural generalization of the well known red-blue matching problem, whose complexity status is unknown: although a randomized polynomial-time algorithm exists, a deterministic algorithm has remained elusive for nearly four decades. 
The best known deterministic approach to the red-blue matching problem, due to Yuster (2012), gives an additive approximation. In this paper, we show a similar result for the red-blue-yellow matching problem, giving a polynomial-time deterministic algorithm that, under natural assumptions, finds a matching satisfying the color requirements almost exactly and has cardinality within 3 of the optimal solution. Our algorithm is a mix of classic linear programming techniques and {\em ad hoc }existence results on restricted classes of graphs such as paths and cycles. As a key ingredient, we prove a curious topological property of plane curves, which is a strengthened version of a result by Grandoni and Zenklusen (2010) in the related context of budgeted matchings.
\end{abstract}

\section{Introduction} \label{sec:intro}

Research on matching problems has been related to computational complexity since the celebrated work of Edmonds \cite{edmonds1965paths}, where the very concept of a polynomial-time algorithm is introduced. A prominent variant of the classical matching problem is the \emph{red-blue matching problem}, also known as \emph{exact matching problem}: given a graph $G$, where each edge is colored red or blue, and an integer $k\ge0$, it asks whether there is a perfect matching of $G$ containing exactly $k$ red edges. First introduced by Papadimitriou and Yannakakis \cite{papadimitriou1982complexity}, this problem gained popularity after Mulmuley, Vazirani, and Vazirani \cite{mulmuley1987matching} gave a randomized (polynomial-time) algorithm for it. (Throughout the paper, all algorithms are meant to be polynomial-time, so we will omit to specify this.) The question of whether the problem admits a {\em deterministic} algorithm has been open for almost four decades, despite a large amount of work \cite{yuster2012almost, gurjar2017exact, el2023exact}.
 This issue is most important due to the popular conjecture that $\mathcal{P}=\mathcal{RP}$, i.e., that all problems that admit a randomized algorithm also admit a deterministic algorithm \cite{kabanets2003derandomizing}.
 
 Much research has been done on variations and extensions of the problem (see for instance \cite{el2023exact, durr2023approximation}), with implications for other areas, such as integer programming: indeed, exact matching can be modeled as an integer program with a coefficient matrix that is \lq\lq nearly'' totally unimodular (see \cite{aprile2025integer}), and randomized algorithms for the weighted version of the exact matching problem, given in \cite{camerini1992random}, lie at the core of algorithms for certain classes of integer programs with bounded subdeterminants \cite{nagele2024congruency, nagele2024advances}. 

While a deterministic algorithm for the red-blue matching problem is yet to be found, a line of research focused on obtaining approximate results: in particular, Yuster \cite{yuster2012almost} considered the problem of finding matchings that have exactly $k$ red edges and are \lq\lq almost'' maximum, 
meaning that they have at most one edge less than the maximum. Other works \cite{grandoni2010approximation, chekuri2011multi, grandoni2014new, doron2023eptas} studied a weighted version of the problem, where, instead of colors, we are given multiple weight functions defined on the edges of the graph, and the solution must comply with the respective budgets. In this case, exact constraints are replaced by inequality constraints, leading to PTASs. 
Our work lies between these two approaches, as we consider the following \emph{RBY (red-blue-yellow) matching problem}: 
\begin{quote}
Given a graph where each edge is colored red, blue, or yellow, and given nonnegative integer numbers $k_R, k_B$, what is the maximum cardinality of a matching with exactly $k_R$ red edges and $k_B$ blue edges?
\end{quote}

Clearly, the above problem generalizes the red-blue matching problem. 
Even more, already the question of whether there exists \emph{any} matching that exactly satisfies the above color requirements generalizes the red-blue matching problem. Hence, while the RBY matching problem can be solved via the randomized algorithm for exact matching (see \cite{camerini1992random}), a deterministic exact approach is for now out of reach. We will however give an algorithm to find a matching that almost satisfies the color requirements (losing at most one blue edge) and is large with respect to the optimum value of the natural linear-programming relaxation of the problem. In particular, we will obtain an {\em additive} approximation guarantee, thus making our result incomparable with the aforementioned PTASs on the (multi-)budgeted version of matching (see, for instance, \cite{chekuri2011multi, grandoni2014new}): indeed, while those support more general constraints than our color requirements, the approximation guarantee is a {\em multiplicative} one.
At the end of this section, we detail further connections between those papers and ours, describing in particular how they can be used to obtain weaker versions of our result.

Before we further describe our contribution, we state the precise result of Yuster \cite{yuster2012almost} which is our starting point. In the following, adopting standard notation from graph theory, we denote by $\alpha'(G)$ the maximum size of a matching in a given graph $G$, without color restrictions.  Moreover, to simplify the exposition, we say \lq\lq red-blue(-yellow) colored graph'' to indicate a graph where each edge is colored red or blue (or yellow). 

\begin{thm}\label{thm:yuster}\cite[Theorem 1.1]{yuster2012almost}
    There is a deterministic algorithm that, given a red-blue colored graph $G$ and an integer $k\ge0$, either correctly asserts that there is no matching of size $\alpha'(G)$ with exactly $k$ red edges, or returns a matching of size at least $\alpha'(G)-1$ with exactly $k$ red edges.
\end{thm}

We remark that this algorithm might return a matching of size $\alpha'(G)-1$ with $k$ red edges even when there is one of size $\alpha'(G)$.

The above result can be strengthened as follows, with the same proof as in \cite{yuster2012almost}. Let us denote by $P_M(G)$ the matching polytope of $G$, i.e., the convex hull of all incidence vectors of matchings of $G$. Throughout the paper, we denote by $E$ the set of edges of $G$, and we write $x(F):=\sum_{e\in F} x_e$ for any $F\subseteq E$. Consider the linear program (LP) $\max \{x(E): x\in P_M(G),\, x(R)=k\}$, where  $R$ is the set of red edges of $G$. Assume that the LP is feasible (otherwise there is of course no matching with exactly $k$ red edges) and has optimal value $\alpha^*$. Then, there is a deterministic algorithm that finds a matching with $k$ red edges and cardinality at least $\lceil\alpha^*\rceil-1$. Let us sketch this algorithm. First, one solves the aforementioned LP (this can be done in polynomial time thanks to the results of Edmonds \cite{edmonds1965maximum}, Padberg and Rao \cite{PadbergRao}, and the equivalence of separation and optimization). If the optimal solution is integer, the algorithm returns the corresponding matching of $G$. Otherwise, the hyperplane $x(R)=k$ cuts an edge of $P_M(G)$ and thus one can obtain an optimal solution lying in the relative interior of an edge of $P_M(G)$, which is therefore a convex combination of (the incidence vectors of) two adjacent matchings of $G$. It can be seen that one of these matchings has at least $k$ red edges and the other has at most $k$ red edges, and their symmetric difference is a path or an even cycle.  By proving a simple result on the almost perfect matchings of paths and even cycles, one finds that it is always possible to select edges from the symmetric difference in order to obtain a matching with exactly $k$ red edges, losing at most one in size. Let us focus on the case where the symmetric difference is a cycle, hence both matchings have the same size $\alpha^*$, as the other cases are similar. The following holds:

\begin{prop}\cite[Lemma 2.2]{yuster2012almost}\label{prop:yuster}
    Let $G$ be a red-blue colored even cycle, and let $M_0$ (resp., $M_1$) be the perfect matching of all even (resp., odd) edges of $G$. Then, for any integer $k$ between $|M_0\cap R|$ and $|M_1\cap R|$ there exists an almost maximum matching of $G$ with exactly $k$ red edges.
\end{prop}

Hence, by applying the above result to the symmetric difference of the two matchings with an appropriate choice of $k$, one obtains the desired matching.

In this paper, we give a version of Yuster's results in the more general setting of red-blue-yellow colored graphs, where we have a requirement on the number of red edges that we want to satisfy exactly, and a requirement on the number of blue edges that we want to satisfy almost exactly. The following is our main result. 

\begin{thm}\label{thm:main}
There is a deterministic algorithm that,
    given a red-blue-yellow colored graph $G$ (thus, with edge set $E=R\mathbin{\dot\cup} B\mathbin{\dot\cup} Y$) and integers $k_R, k_B\ge0$ such that the LP \[\max\{x(E): x\in P_M(G),\, x(R)=k_R,\, x(B)=k_B\}\] 
     is feasible and has optimal value $\alpha^*$, finds a matching $M$ of $G$ with $|M|\geq \lfloor \alpha^*\rfloor-3$, exactly $k_R$ red edges, and $k_B$ or $k_B-1$ blue edges. 
\end{thm}

Similarly to the algorithm of Theorem \ref{thm:yuster} described above, our algorithm starts by solving the LP in the statement. Again, there is an optimal solution that lies in the relative interior of a low-dimensional face of $P_M(G)$, hence we obtain a small number of matchings to deal with. The union (or symmetric difference) of these matchings can be in general hard to study. Indeed, if, for example, the matchings are three and disjoint, 
one obtains a 3-regular graph that is uniquely 3-edge colorable, a mysterious class of graphs that has connections to the circuit double cover conjecture and other open problems (see \cite{zhang1995hamiltonian, zhang2001strong}). To the best of our knowledge, nothing is known on the existence of large matchings in such graphs, or in general colored 3-regular graphs.

However, we will be able to restrict ourselves to pairs of matchings, and in particular we will only need to study almost perfect matchings of a (red-blue-yellow colored) cycle or path. We will focus on the case of the cycle, as it easily implies the other cases. Proving an extension of Proposition \ref{prop:yuster} (whose proof is rather simple) to red-blue-yellow colored cycles will be a crucial step in the proof of Theorem \ref{thm:main}, requiring substantial work.

\begin{figure}
    \centering
    \begin{tikzpicture}
        \node[circle, fill= black, inner sep = 1.5] (1) at (0,0) {};
        \node[circle, fill= black, inner sep = 1.5] (2) at (0,1) {};
        \node[circle, fill= black, inner sep = 1.5] (3) at (1,2) {};
        \node[circle, fill= black, inner sep = 1.5] (4) at (2,2) {};
        \node[circle, fill= black, inner sep = 1.5] (5) at (3,1) {};
        \node[circle, fill= black, inner sep = 1.5] (6) at (3,0) {};
         \node[circle, fill= black, inner sep = 1.5] (7) at (2,-1) {};
          \node[circle, fill= black, inner sep = 1.5] (8) at (1,-1) {};

          \draw[red, very thick] (1) -- (2);
          \draw[blue, very thick] (3) -- (2);
          \draw[yellow, very thick] (3) -- (4);
          \draw[blue, very thick] (4) -- (5);
          \draw[red, very thick] (5) -- (6);
          \draw[blue, very thick] (6) -- (7);
          \draw[yellow, very thick] (7) -- (8);
          \draw[blue, very thick] (1) -- (8);
    \end{tikzpicture}
    \caption{In the cycle above, a matching with two red edges cannot have any blue edges.}
    \label{fig:cycle}
\end{figure}
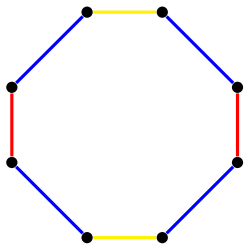

In order to appreciate the difference, we first note that, when we have requirements on two colors, a straightforward extension of Proposition \ref{prop:yuster} (i.e., the existence of an almost maximum matching with $k_R$ red edges and $k_B$ blue edges) fails.
Indeed, consider a red-blue-yellow colored cycle  and let $k_R$ be between $|M_0\cap R|$ and $|M_1\cap R|$, and similarly, let $k_B$ be between $|M_0\cap B|$ and $|M_1\cap B|$. Is there always an almost maximum matching of the cycle with exactly $k_R$ red edges and $k_B$ blue edges? The answer is negative: 
for instance, in the cycle in Figure \ref{fig:cycle}, one perfect matching (say, the one containing the even edges) has four blue edges and the other has two red edges (and no blue edge), but there is no matching with two red edges and one blue edge. This leads to the question of which further conditions could be necessary to guarantee the existence of an almost maximum matching. 
Indeed, the colored cycles arising as solutions to our LP are not general, but enjoy a convexity property relating the colors of odd edges, even edges, and the numbers $k_R, k_B$. A similar property is exploited in \cite{grandoni2010approximation} in the context of budgeted matchings: there, a nice lemma about plane curves 
ensures the existence of matchings of large cardinality whose weight does not exceed a given budget. We state the lemma here, as it will be the starting point of our result on cycles.

Let $f:  [0,\tau]\rightarrow \R^2$ be a continuous piecewise-linear map, hence $\Im(f)$ is a piecewise-linear (i.e., polygonal) plane curve. (In the paper, we will often identify a map $f$ as above with the corresponding curve $\Im(f)$.)
We define $f^\infty:\R\to\R$ as the extension of $f$ by periodicity: that is, for every $t\in\R$,
\begin{equation}\label{eq:f^inf}
f^\infty(t)=k(f(\tau)-f(0))+f(r),\:\mbox{if $t=k\tau+r$ with $k\in\Z$ and $0\le r<\tau$}.
\end{equation}
Furthermore, given $q\in\R^2$, we write $f+q$ to denote the map $t\mapsto f(t)+q$ for $t\in[0,\tau]$.
The following is a restatement of Lemma 1 in \cite{grandoni2010approximation}, which easily follows from its proof.

\begin{lemma}\cite[Lemma 1]{grandoni2010approximation}\label{lem:curves}
     Let $f: [0,\tau]\rightarrow \R^2$ be a continuous piecewise-linear map, and let $q$ be a point in the open segment between $f(0)$ and $f(\tau)$.  Then the curves $f^\infty$ and $f+q-f(0)$ intersect: in particular, there exist $u\in [0,2\tau]$ and $v\in [0,\tau]$ with $v< u < v+\tau$ such that $f^\infty(u) = f(v)+q-f(0)$. 
\end{lemma}

Given two continuous maps $f,g:[0,\tau]\to\R^2$, a pair $(u,v)\in\R\times[0,\tau]$ such that $f^\infty(u) = g(v)$ will be called an \emph{intersecting pair} for $(f^\infty,g)$. Thus, the above result states the existence of an intersecting pair for $(f^\infty,f+q-f(0))$ with the additional property $v<u<v+\tau$.

When reading the proof of the above lemma in \cite{grandoni2010approximation}, which is based on Jordan curve theorem, it seems reasonable to believe that not only must the two curves intersect, but they must also \emph{cross}.  The notion of two plane curves crossing appears to be intuitive, but is actually quite subtle and requires some carefulness (and assumptions), especially when one wants to define {\em where} a crossing happens (see Section \ref{sec:crossing} for a brief discussion about this). For our purposes, besides assuming that $f$ and $g$ are  continuous and piecewise linear, we can also suppose that $f^\infty$ is injective. Then, as we will verify in Section \ref{sec:crossing} again using Jordan curve theorem, $f^\infty$ divides $\R^2$ into two topologically connected components (meaning that $\R^2\sm\Im(f^\infty)$ consists of two connected components). We say that a pair $(u,v)\in\R\times(0,\tau)$ is a {\em crossing pair} for $(f^\infty,g)$ if:
\begin{itemize}
\item $f^\infty(u)=g(v)$;
\item  there exists $s\in(0,v]$ and $\epsilon>0$ such that
\begin{itemize}
    \item $g([s,v])\subseteq f^\infty(\R)$,  \item $g((s-\epsilon,s))$, $g((v,v+\epsilon))$ are contained in opposite \lq\lq sides'' of $f^\infty$ (i.e., in different components of $\R^2\setminus\Im(f^\infty)$). 
\end{itemize}
\end{itemize}
When the above conditions are met, we say that $g$ crosses $f^\infty$. More generally, if $(u,v)$ is a crossing pair and $I\subseteq\R,J\subseteq[0,\tau]$ are real intervals such that $u$ is in the interior of $I$ and $v$ is in the interior of $J$, we say that $g|_J$ crosses $f^\infty|_I$.

Note that $g$ is allowed to overlap with $f^\infty$ at more than one point: when this happens, i.e., when $s<v$, we say that the crossing is {\em overlapping}; otherwise, when $s=v$, the crossing is {\em simple}.
We also remark that a crossing pair $(u,v)$ for $(f^\infty,g)$ must satisfy $0<v<\tau$ by definition. Thus, throughout the paper, we will often omit to specify this condition.


In Section \ref{sec:crossing} we will show the following:

\begin{lemma}[Crossing Lemma]\label{lem:crossing}
Let $f: [0,\tau]\rightarrow \R^2$ be a continuous piecewise-linear map such that $f^\infty$ is injective, and let $q$ be a point on the segment between $f(0)$ and $f(\tau)$ such that $q\not\in\Im(f)$. 
Then there exists a crossing pair $(u,v)$ for $(f^\infty,f+q-f(0))$ such that $v<u<v+\tau$.
\end{lemma}

\begin{figure}[h]
    \centering     \centering \begin{tikzpicture}[scale=0.5]

    \draw[step=1, lightgray, very thin] (0,0) grid (9,7);

    \draw[->] (0,0) -- (9.5,0) node[right] {};
    \draw[->] (0,0) -- (0,7.5) node[above] {};

\draw[very thick] (1,1) -- (2,3) -- (5,3);

\draw[ thick, gray] (5,3) -- (6,5)-- (8,5);
\draw[ thick, gray, dotted] (9,5)-- (8,5);
\draw[ thick, gray, dotted] (0,1)-- (1,1);

\node[circle,fill=black, inner sep=1] (a) at  (1,1){};
\node[circle,fill=black, inner sep=1] (b) at  (5,3){};

\node[circle,fill=red, inner sep=1.2] (q) at  (3,2){};
\node[right] at (q) {$q$};

\draw[red, thick] (q) -- (4,4) -- (7,4);

\newcommand{\radius}{0.2}
\draw (3.5-\radius,3-\radius) rectangle (3.5+\radius,3+\radius);

\draw (5.5-\radius,4-\radius) rectangle (5.5+\radius,4+\radius);

\end{tikzpicture}\hspace{1.2cm}
      \begin{tikzpicture}[scale=0.5]

    \draw[step=1, lightgray, very thin] (0,0) grid (9,7);

    \draw[->] (0,0) -- (9.5,0) node[right] {};
    \draw[->] (0,0) -- (0,7.5) node[above] {};

\begin{scope}[shift={(-1,0)}]

\draw[very thick] (2,3) -- (3,5) -- (5,1)-- (6,3);

\draw[ thick, gray] (6,3) -- (7,5)-- (9,1) -- (10,3);
\draw[ thick, gray] (1,1)-- (2,3);

\node[circle,fill=black, inner sep=1] (a) at  (2,3){};
\node[circle,fill=black, inner sep=1] (b) at  (6,3){};

\node[circle,fill=red, inner sep=1.2] (q) at  (5,3){};
\node[right] at (q) {$q$};

\draw[red, thick] (q) -- (6,5) -- (8,1)-- (9,3);
\end{scope}

\newcommand{\radius}{0.2}
\draw (5.5-\radius,4-\radius) rectangle (5.5+\radius,4+\radius);

\draw (7.5-\radius,2-\radius) rectangle (7.5+\radius,2+\radius);
\end{tikzpicture}
   
    \caption{Examples of crossing pairs, with $f$ shown in black, its extension $f^\infty$ in gray, and $f+q-f(0)$ in red. Squares indicate crossings.
    }
    \label{fig:crossing}
\end{figure}
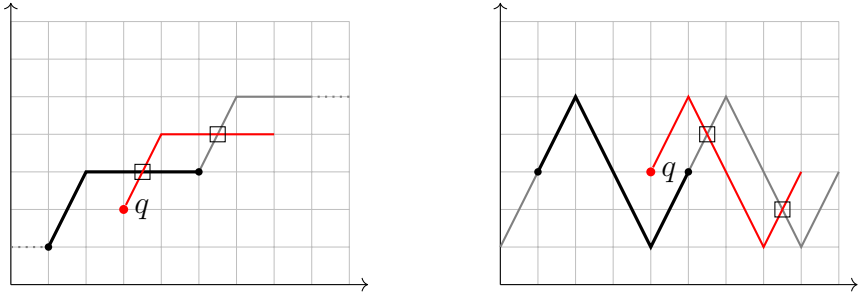

Figure \ref{fig:crossing} shows some examples of crossing pairs satisfying the conditions of Lemma \ref{lem:crossing}.

Lemma \ref{lem:crossing} will be crucial in showing the extension of Proposition \ref{prop:yuster} (see Theorem \ref{thm:cycle_choice} in Section \ref{sec:cycle}) that we need to prove Theorem \ref{thm:main}. We also believe that this lemma could be used in other contexts to prove the existence of combinatorial objects with special properties.

We conclude this section with two remarks on how existing results, and in particular Lemma \ref{lem:curves}, can lead to weaker forms of our Theorem \ref{thm:main}. 
\begin{itemize}
    \item The first weaker guarantee follows from an application of the algorithm for multi-budgeted matching problems described in \cite[Section 4.3]{grandoni2014new}. For each possible number of yellow edges in an optimal solution, an instance of the matching problem with one budget per color is defined and that algorithm is run.  The algorithm achieves an additive error of 13 when there are three budgets. Applied to our instance, one can show that it yields a matching whose size is within 13 of the optimal value and whose violations of the color requirements sum up to at most 13. We remark that our algorithm, apart from having a much better guarantee, is also more efficient, as it does not need guessing and hence only solves one LP (as in the statement of Theorem \ref{thm:main}) instead of many.
    \item A better guarantee comes from a refinement of the arguments in \cite{grandoni2014new} (also appearing in \cite{grandoni2010approximation}), in particular the application of Lemma \ref{lem:curves}. As we detail in Section \ref{sec:cycle}, by defining appropriate curves and finding intersecting pairs for them, we obtain a weaker version of our result on cycles. This, in turn, implies a weaker version of our Theorem \ref{thm:main}, where we allow a loss of two on the required number of blue edges instead of one (see the discussion at the end of Section \ref{sec:goodpaths}). It will require applying the Crossing Lemma, together with additional case analysis, to prove our result.
\end{itemize} 

\paragraph{Outline} The paper is organized as follows. Section \ref{sec:cycle} contains the aforementioned result on cycles and paths, and its proof, which exploits the Crossing Lemma (Lemma \ref{lem:crossing}). In Section \ref{sec:general}, we show how this leads to the proof of our main result (Theorem \ref{thm:main}). Section \ref{sec:crossing} contains the proof of the Crossing Lemma. Finally, in Section \ref{sec:final} we discuss some open questions and possible extensions of our results, in particular to the case of more than three colors.

\bigskip

\section{Cycles and paths}\label{sec:cycle}

In this section, we consider a colored graph $G=(V,E)$, whose edge set $E$ is partitioned into three subsets $R, B, Y$, (red, blue, and yellow edges).
We will privilege two colors, say red and blue, and choose them as coordinates to represent matchings in the plane: in particular, with every matching $M$ of $G$, we associate a point $p_M=(|R\cap M|, |B\cap M|)\in \R^2$.

The goal of this section is to show the following theorem. Before stating it, we remark that whenever we consider a path or a cycle, we assume that the edges are numbered starting from zero: in the case of a path, edge 0 is one of the extremal edges of the path, whereas, for a cycle, it is an arbitrary edge, and the edges are numbered consecutively along the path or the cycle.

\begin{thm}\label{thm:cycle_paths_choice}  Let $G$ be a colored path or colored even cycle, and denote by $M_0, M_1$ the matchings of even and odd edges respectively, with $|M_0|\geq |M_1|$. 
Assume that $(k_R, k_B)$ is an integer point on the segment between $p_{M_0}, p_{M_1}$. Then there is a matching $M$ of $G$ with $|M|\geq |M_1|-1$, exactly $k_R$ red edges, and $k_B$ or $k_B-1$ blue edges. This matching can be found efficiently.
\end{thm}


Theorem \ref{thm:cycle_paths_choice} is actually a consequence of the following result, which only concerns even cycles.

\begin{thm}\label{thm:cycle_choice} 
    Let $G$ be a colored even cycle of length $2\ell$ and denote by $M_0, M_1$ the perfect matchings of even and odd edges, respectively.
    Assume that $(k_R, k_B)$ is an integer point on the segment between $p_{M_0}, p_{M_1}$. Then:
    \begin{enumerate}
        \item If  $Y\neq \emptyset$, there is a matching $M$ of $G$ with $|M|\ge \ell-1$ and $p_M =(k_R, k_B)$. 
        \item If $Y=\emptyset$, there is a matching $M$ of $G$ with $|M|\ge \ell-1$ and $p_M =(k_R, k_B-1)$.  
    \end{enumerate}
\end{thm}

We first show how Theorem \ref{thm:cycle_choice} implies Theorem \ref{thm:cycle_paths_choice}.

\begin{proof}[Proof of Theorem \ref{thm:cycle_paths_choice}]
 
If $G$ is an even cycle, Theorem \ref{thm:cycle_choice} clearly implies the existence of the desired matching.
    If $G$ is a path of even length, we identify the two extremes of $G$, thus obtaining an even cycle. By applying Theorem \ref{thm:cycle_choice}, we obtain a matching that is also matching for $G$ and satisfies all the requirements. 

    Now, let $G$ be an odd path, say an $M_0$-augmenting path of length $2\ell-1$. We add a dummy yellow edge joining the extremes of the path, in order to turn it into an even cycle, and apply Theorem \ref{thm:cycle_choice}, thus obtaining a matching $M$ of size at least $\ell-1$. Possibly removing the dummy edge from $M$, we find the desired matching with at least $\ell-2 = |M_1|-1$ edges. 

    Finally, in each of these cases a matching with the properties of Theorem \ref{thm:cycle_paths_choice} can be found efficiently by inspection, as every matching of size at least $|M_1|-1$ leaves exposed at most two nodes. 
\end{proof}

\begin{remark}\em
  The bound on $|M|$ given in Theorem \ref{thm:cycle_paths_choice} is tight. Consider, for instance, the odd path whose color sequence is B\,R\,Y\,R\,Y\,B\,Y\,B\,Y\,R\,Y\,R\,B, and choose $k_R=3, k_B=2$. Here $|M_0|=7$, $|M_1|=6$, $p_{M_0}=(0,2)$, and $p_{M_1}=(4,2)$, thus $(k_R,k_B)$ belongs to the segment between $p_{M_0},p_{M_1}$. However, the only matchings with three red edges and two blue edges have no yellow edges, and the only matchings with three red edges and one blue edge have at most one yellow edge. Thus, these matchings have size at most 5. This shows that the bound of Theorem \ref{thm:cycle_paths_choice} is tight for odd paths.
By attaching a yellow edge to the last node of the path above, we obtain an even path for which the bound of the theorem is tight.
The same holds if we turn the even path into a cycle by identifying its extreme nodes.
\end{remark}


\subsection{Proof of Theorem \ref{thm:cycle_choice}: preliminary results}
The rest of the section is dedicated to proving Theorem \ref{thm:cycle_choice}.
We begin with a brief proof overview. In this subsection, we observe that the theorem is easier to prove in certain cases. We therefore rule out these cases by formulating a series of assumptions, which we may impose throughout the argument. These assumptions simplify both the notation and the structure of the proof. In Section \ref{sec:goodpaths}, we introduce the notions of good paths and imbalance curve, which are crucial ingredients for the proof. By exploiting the existence of intersecting pairs related to our imbalance curve, we will make a first step toward proving Theorem \ref{thm:cycle_choice}, showing a weaker version already mentioned at the end of Section \ref{sec:intro}. Finally, in Section \ref{sec:cycle_crossing}, we exploit additional structural properties, such as the existence of crossing pairs, and complete the proof  through a case analysis.

In the following we assume that the hypotheses of \ref{thm:cycle_choice} are satisfied. In particular, $G$ is a colored cycle whose edges are numbered from $0$ to $2\ell-1$, for some integer $\ell\ge 2$.

A first easy observation is that Theorem \ref{thm:cycle_choice} holds if $(k_R,k_B)$ coincides with $p_{M_0}$ (resp., $p_{M_1}$,), as in this case it is sufficient to choose $M=M_0$ (resp., $M_1$). Therefore, in the rest of the section we will work under the following assumption:

\begin{ass}\label{ass:open}
The integer point $(k_R,k_B)$ belongs to the open segment between $p_{M_0},p_{M_1}$.
\end{ass}

Now, let $P$ be a nontrivial even path in $G$, or $G$ itself. We say that $P$ is \emph{balanced} if each color appears the same number of times on the even edges of $P$ and on the odd edges of $P$, i.e., $|P\cap C\cap M_0| = |P\cap C\cap M_1|$ for each $C\in\{R,B,Y\}$. (We often abuse notation and use $P$ to refer to the edge set of $P$.) Two incident edges of the same color are an example of a balanced path. To simplify the exposition, we call $P$ a balanced {\em path} even when $P=G$.

\begin{lemma}\label{lem:balanced}
    In proving Theorem \ref{thm:cycle_choice}, it is sufficient to consider cycles that do not contain balanced paths.
\end{lemma}
\begin{proof}
Assume by contradiction that Theorem \ref{thm:cycle_choice} holds for cycles without balanced paths, but there is some cycle $G$ containing a balanced path $P$ that violates the theorem. We choose $G$ of minimum length, and, after fixing $G$, we also take $P$ of minimum length. Note that $P$ has an even number of edges. 

By Assumption \ref{ass:open}, $p_{M_0}\ne p_{M_1}$, i.e., $G$ is not balanced. Thus, $P\ne G$.
Furthermore, we claim that $|P|\le|G|-4$. Indeed, if $|P|=|G|-2$, the balancedness of $P$ implies that $p_{M_1}-p_{M_0}\in\{0,\pm1\}^2$. Then the only integer points on the segment between $p_{M_0},p_{M_1}$ are $p_{M_0},p_{M_1}$ themselves. This contradicts Assumption \ref{ass:open}, and therefore $|P|\le|G|-4$.

Let $G'$ be obtained from $G$ by contracting $P$ to a single vertex. By the above observations, $G'$ is a cycle with an even number of edges. 
We analyze three cases.

\begin{enumerate}
\item
Assume first that $G$ has no yellow edges. In this case, it is easy to see that $P$ must have length two: if not, by the minimality of $P$, there are no pairs of consecutive red-red or blue-blue edges in $G$ (RR and BB repetitions, for short), thus all even edges are red and all odd edges are blue (or the other way around), and there is no balanced path, a contradiction. Therefore $P$ has length two, and to fix ideas we assume that $P$ consists of two consecutive red edges. Then $G'$ is a cycle of length $2\ell-2$ that satisfies the hypotheses of Theorem \ref{thm:cycle_choice} with $(k_R,k_B)$ replaced by $(k_R-1, k_B)$. 
By the minimality of $G$, the theorem holds for $G'$ and thus there is a matching $M'$ in $G'$ of size $\ell-2$ such that $p_{M'}=(k_R-1, k_B-1)$. In $G$, $M'$ is a matching that leaves exposed at least one of the extreme nodes of $P$. We can then augment $M'$ with one of the edges of $P$, thus obtaining a matching $M$ of size $\ell -1$ such that $p_{M}=(k_R, k_B-1)$. This contradicts the assumption that $G$ violates Theorem \ref{thm:cycle_choice}.
\item 
Assume now that $G'$ (and hence $G$) has at least one yellow edge.
Since $P$ is balanced, both its odd and its even edges have $h_R$ red and $h_B$ blue edges for some $h_R, h_B\in \Z_{\geq 0}$. Hence, $G'$ satisfies the hypotheses of Theorem \ref{thm:cycle_choice} with $p_{M_0}, p_{M_1}, (k_R,k_B)$ all decreased by $(h_R, h_B)$. Given a matching $M'$ of $G'$ with $|M'|=h-1$ (where $2h$ is the length of $G'$) and $p_{M'}=(k_R-h_R, k_B-h_B)$, a matching $M$ of $G$ with $|M|=\ell-1$ and $p_M=(k_R,k_B)$ is obtained by adding to $M'$ either the odd or the even edges of $P$.
\item
We are left with the case in which $G$ has yellow edges but $G'$ has none. In this case, rather than exploiting Theorem \ref{thm:cycle_choice} for $G'$, we prove it directly for $G$. Note that $P$ contains yellow edges, thus, due to the minimality of $P$, $G$ has no BB or RR repetition. As $G'$ has even length and contains no yellow edges, this implies that the edges of $G'$ alternate between red and blue.

Now, let us renumber the edges\footnote{Renumbering the edges might swap the role of $M_0$ and $M_1$, but this does not affect the hypotheses of Theorem \ref{thm:cycle_choice}.} of $G$ starting from the last yellow edge of $P$: that edge is numbered 0, and it is followed by a sequence of alternating red and blue edges that contains all the edges of $G'$. 
To fix ideas, say that edge $1$ is red, edge 2 is blue, and so forth, as the other case is analogous.
If $G'$ has $2h$ edges, and we define $h_R,h_B$ as in the previous case, we have that
$p_{M_0}=(h_R, h_B+h)$ and $p_{M_1}=(h_R+h, h_B)$. By Assumption \ref{ass:open}, there exists $\lambda\in (0,1)$ such that $(k_R, k_B)=\lambda p_{M_0} + (1-\lambda)p_{M_1},$ implying
\begin{equation}\label{eq:kRkB}
    (k_R, k_B)=(h_R+(1-\lambda)h, h_B + \lambda h),     
\end{equation}
from which it follows that $\lambda h$, $(1-\lambda) h$ are integer numbers.
    
Define now the following sequence of matchings:  $M^{(0)}=M_0\setminus\{0\}$, $M^{(i)}=M^{(i-1)}\cup \{2i-1\}\setminus \{2i\}$ for $i\in\{1,\dots,h\}$. Notice that all these matchings have size $\ell-1$. We have $p_{M^{(0)}}=p_{M_0}=(h_R, h_B+h)$ because edge 0 is yellow, and then the sequence proceeds by exchanging one red edge for a blue edge. In particular, because of \eqref{eq:kRkB}, $M^{((1-\lambda)h)}$ has $k_R$ red edges and $k_B$ blue edges, which contradicts the assumption that $G$ violates Theorem \ref{thm:cycle_choice}.    
\end{enumerate}
This concludes the proof.
\end{proof}

Because of the above lemma, from now on, in proving Theorem \ref{thm:cycle_choice} we can make the following assumption:

\begin{ass}\label{ass:no_balanced}
     $G$ does not contain balanced paths. In particular, $G$ does not contain two consecutive edges of the same color, i.e., the coloring is \emph{proper} (borrowing the term from classical edge coloring). 
\end{ass}

Lemma \ref{lem:balanced} (and in particular Assumption \ref{ass:no_balanced}) easily implies the second case of Theorem \ref{thm:cycle_choice}, i.e.,  when there are no yellow edges. This red-blue case was already shown in \cite{yuster2012almost} (as mentioned in Section \ref{sec:intro}), but enforcing the coloring to be proper makes the proof immediate, as demonstrated below.

\begin{obs}\label{obs:noY}
  Theorem \ref{thm:cycle_choice} holds under the additional assumption that $Y=\emptyset$.
\end{obs}

\begin{proof}
   If $Y=\emptyset$, Assumption \ref{ass:no_balanced} implies that the edges of $G$ alternate between red and blue: let us say that all edges of $M_0$ are red and all edges of $M_1$ are blue. Since $(k_R,k_B)$ is a convex combination of $p_{M_0},p_{M_1}$, we have $k_R+k_B=\ell$. Then, the desired matching is obtained by selecting the first $k_R$ edges of $M_0$ and all the blue edges not incident with them, whose number is $\ell-k_R-1=k_B-1$.
\end{proof}

Thus, from now on we can focus on the first case of Theorem \ref{thm:cycle_choice}:

\begin{ass}\label{ass:Y}
   $G$ contains at least one yellow edge, i.e., $Y\neq \emptyset$.
\end{ass}

\subsection{Good paths and the imbalance curve}\label{sec:goodpaths}

Given a path $P$ of $G$ (or $P=G$), let $P_0=P\cap M_0$ (resp., $P_1=P\cap M_1$) be the set of its even (resp., odd) edges. We define the \emph{imbalance} of $P$ as the point
\[\delta(P) =(|P_1\cap R|-|P_0\cap R|, |P_1\cap B|-|P_0\cap B|)\in\R^2.\]
Notice that an even path is balanced if and only if $\delta(P)=(0,0)$. 

Now, let $G$ and $(k_R,k_B)$ be as in Theorem \ref{thm:cycle_choice}. For each $i\in\{0,1\}$, define $r_i=|M_i\cap R|$ and $b_i=|M_i\cap B|$.
Also, let $q=(k_R-r_0, k_B-b_0)$. We say that a nontrivial even path $P$ of $G$ is \emph{good} if it starts at an even edge and satisfies $\delta(P)=q$. 
\lq\lq $P$ starts at an even edge'' means that its edge numbers are of the form $2v,\dots, 2u-1$ with $v< u<v+\ell$, where, here and in the sequel, the numbering $2v,\dots, 2u-1$ must be taken modulo $2\ell$.  

In order to show Theorem \ref{thm:cycle_choice}, we will show the existence of paths that are good and also have some useful extra properties. 
To this end, let us now define a piecewise linear curve $d: [0,\ell]\rightarrow \R^2$, as follows:
\[d(t)= \delta(\{0,\dots,2t-1\}) \mbox{ for every $t\in \{0,\dots,\ell\}$},\]
and the curve is obtained by linearly interpolating the points $d(0),\dots,d(\ell)$. 
We call $d$ the {\em imbalance curve} of $G$. Let us give some intuition on the crucial role of curve $d$, which tracks the imbalance of the initial portions of our cycle. %
Notice that $d(0)=(0,0)$,  $d(\ell)=p_{M_1}-p_{M_0}$, and the point $q=(k_R-r_0, k_B-b_0)$ is on the segment between $d(0),d(\ell)$. Moreover, if $d$ contains $q$, in particular if $d(t)=q$ for some integer $t$, then by definition of $d$ the path from edge 0 to edge $2t-1$ has imbalance $q$, i.e., it is a good path. More generally, we can obtain a good path by finding integers $u, v$ with $u>v$ such that $d(u)-d(v)=q$. As we will verify, this corresponds to looking for an intersecting pair for $(d^\infty,d+q)$ (see Figure \ref{fig:d_example} for an intuition).
Hence, in the following we will study the behavior of $d$ in order to show the existence of the desired intersecting (actually, crossing) pair and good path.

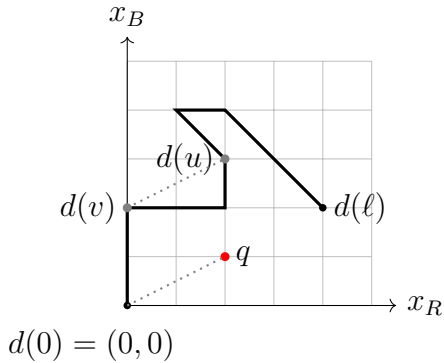
\begin{figure}[h]
    \centering         \begin{tikzpicture}[scale=0.65]

    \draw[step=1, lightgray, very thin] (0,0) grid (5,5);

    \draw[->] (0,0) -- (5.5,0) node[right] {$x_R$};
    \draw[->] (0,0) -- (0,5.5) node[above] {$x_B$};


\node[right] at (4,2) {$d(\ell)$};

\node[left] at (1.2,-0.8) {$d(0)=(0,0)$};

\draw[very thick] (0,0) -- (0,2) -- (2,2) -- (2,3) -- (1,4) -- (2,4) -- (3,3) -- (4,2);



\node[circle,fill=black, inner sep=1] (dl) at  (4,2){};
\node[circle,fill=black, inner sep=1] (d0) at  (0,0){};


\node[circle,fill=red, inner sep=1.2] (q) at  (2,1){};
\node[right] at (2,1) {$q$};

\node[circle,fill=gray, inner sep=1.2] (v) at  (0,2){};
\node[left] at (0,2) {$d(v)$};

\node[circle,fill=gray, inner sep=1.2] (u) at  (2,3){};
\node[left] at (2,3) {$d(u)$};

\draw[gray, thick, dotted]  (0,2) -- (2,3);
\draw[gray, thick, dotted]  (0,0) -- (2,1);

\end{tikzpicture}
    \caption{The imbalance curve $d: [0,\ell]\rightarrow \R^2$, with $\ell=9$, corresponding to the 18-edge cycle with the following coloring (where the edges are numbered from 0 to 17): Y\,B\,Y\,B\,Y\,R\,Y\,R\,Y\,B\,R\,B\,Y\,R\,B\,R\,B\,R.
    The point $q$ (in red) belongs to the segment between $d(0),d(\ell)$. While $d$ does not contain $q$, it does contain several pairs of points whose difference is exactly $q$, in particular $d(u)$ and $d(v)$. As the reader can easily verify, this translates into the existence of intersecting pairs for $(d^\infty,d+q)$ (some of which are crossing pairs).}
    \label{fig:d_example}
\end{figure}

 In the rest of the paper, we will always see $d$ (and every other related curve) as an oriented curve, where the orientation is in the direction of increasing $t$.

We remark that
\begin{equation}\label{eq:d-delta}
d(t)-d(t-1)=\delta(\{2t-2,2t-1\}) \mbox{ for every $t\in\{1,\dots,\ell\}$}.
\end{equation}
This vector depends exclusively on the color of the edges $2t-2, 2t-1$, and can only assume the values $(0,\pm 1)$,$(\pm 1, 0)$, $(-1,1)$, $(1, -1)$. (The value $(0,0)$ is ruled out because of Assumption \ref{ass:no_balanced}.)
Table \ref{tab:arrows} describes the correspondence between the segments of $d$ and the color of the edges of the cycle $G$, where the arrows indicate the orientation of the curve. We refer to the segments of $d$ as {\em moves} of $d$. 
The following observation is straightforward given the above considerations on the moves of $d$. (Here and in the following, by \lq\lq translation of $d$'' we mean the curve $t\mapsto d(t)+u$ for some $u\in\R^2$.)

\begin{obs}\label{cl:d_int}
        One has $d(t)\in \Z^2$ if and only if $t$ is integer, and every breakpoint\footnote{Given a piecewise-linear continuous function $f:I\to\R^2$ defined on a nondegenerate interval $I\subseteq\R$, we call {\em breakpoints} of $f$ the values $t$ in the interior of $I$ where $f$ is not differentiable.}   of $d$ is integer.
    Let $\bar d$ be any translation of $d$ by an integer vector. If $d$ and $\bar d$ intersect (i.e. $d(u)=\bar d(v)$ for some $u,v\in [0,\ell]$), then they intersect at an integer point (hence $u,v\in \Z$).
\end{obs} 

\begin{table}
    \centering
\begin{tabular}{c|c|c|c}
$d(t)-d(t-1)$ & $2t-2$ & $2t-1$ &  \\
\hline
$(-1,\phantom+0)$ & $R$ & $Y$ & $\leftarrow$\\
$(\phantom+1,\phantom+0)$ & $Y$ & $R$ & $\rightarrow$\\
$(\phantom+0,-1)$ & $B$ & $Y$ &\BY\\
$(\phantom+0,\phantom+1)$ & $Y$ & $B$ &\YB\\
$(-1,\phantom+1)$ & $R$ & $B$ & \RB\\
$(\phantom+1,-1)$ & $B$ & $R$ & \BR 
\end{tabular}
\caption{Correspondence between moves of $d$ and edge colors.}
  \label{tab:arrows}
\end{table}

The following lemma makes the connection between good paths and intersecting pairs explicit.

\begin{lemma}    \label{lem:intersecting_pair_good_path}
Let $u, v$ be integers such that $0<v<\ell$, $v < u < v+\ell$, and $d^\infty(u) = d(v)+q$. Then the path $P=\{2v,\dots, 2u-1\}$, where the numbering is modulo $2\ell$, is good.
\end{lemma}
\begin{proof}
    First, consider the case $u\leq \ell$. Then $d(v)+q=d^\infty(u) =d(u)$, i.e., $q=d(u) - d(v)$, which, by definition of $d$, is exactly the imbalance of $P$; hence $P$ is good.
    
    Consider now the case $u>\ell$. 
    Then $d(v)+q=d^\infty(u) = d(u-\ell) + d(\ell)$, and furthermore, $u-\ell < v$. Hence, by applying again the definition of $d$, we have:
    \[q = d(\ell)-(d(v)-d(u-\ell) ) = \delta(\{0,\dots, 2\ell-1\}) - \delta(\{2(u-\ell),\dots,2v-1\}) = \delta(P).\qedhere\]
\end{proof}

Notice that Lemma \ref{lem:curves} guarantees the existence of an intersecting pair $(u,v)$ for $(d^\infty, d+q)$, and Observation \ref{cl:d_int} guarantees that $u,v$ are integer. 
Hence, good paths always exist. The next theorem shows that this already brings us very close to proving Theorem \ref{thm:cycle_choice}. 

\begin{thm}\label{thm:cycles_no_choice}
   Let $G$ and $(k_R,k_B)$ be as in Theorem \ref{thm:cycle_choice}, and assume that $P$ is a good path.  Then there is a matching $M$ of size $\ell-1$ with exactly $k_R$ red edges and $k_B$ or $k_B-1$ blue edges.
\end{thm}
\begin{proof}
Let $P=\{2v,\dots, 2u-1\}$ with $v< u< v+\ell$. 
Let us denote by $r_0^{\rm in}$, $r_1^{\rm in}$ the number of even and odd red edges of $P$, respectively, and by $r_0^{\rm out}$, $r_1^{\rm out}$ the number of even and odd red edges outside $P$. Define similarly $b_0^{\rm in}$, $b_1^{\rm in}$, $b_0^{\rm out}$, $b_1^{\rm out}$. Considering the ``red'' component, we have
\[
k_R-r_0 = q_R=r_1^{\rm in}-r_0^{\rm in},
\]
where the second equality follows from the fact that $P$ is a good path and thus satisfies $\delta(P)=q$. As $r_0=r_0^{\rm in} + r_0^{\rm out}$, we obtain
$r_1^{\rm in}+r_0^{\rm out} = k_R.$ 
Similarly, $b_1^{\rm in}+b_0^{\rm out} = k_B$.
Thus, the set $M$ consisting of the odd edges in $P$ and the even edegs not in $P$ has $k_R$ red edges and $k_B$ blue edges. Working modulo $2\ell$, we can write $M=\{2v+
1,2v+3,\dots, 2u-1\}\cup \{2u,2u+2, \dots, 2v-2\}$. Note that $|M|=\ell$, and $M\setminus\{2u-1\}$, $M\setminus\{2u\}$ are both matchings. (We call such a set $M$ a \emph{quasi-matching}.) Hence, by removing from $M$ one of the edges $2u-1$, $2u$ that is not red (there exists at least one by Assumption \ref{ass:no_balanced}), we obtain a matching of size $\ell-1$ with exactly $k_R$ red edges and $k_B$ or $k_B-1$ blue edges.
\end{proof}

We remark that Theorem \ref{thm:cycles_no_choice} gives a weaker version of Theorem \ref{thm:cycle_choice}, where we allow our matching to have one less blue edge than required also when $Y\ne\emptyset$. One easily sees that this is still sufficient to prove Theorem \ref{thm:cycle_paths_choice}. 
Moreover, this weaker version of Theorem \ref{thm:cycle_choice} leads to a weaker version of our main result (Theorem \ref{thm:main}), where we obtain a matching with at most two fewer blue edges than required (i.e., $k_B-2$, $k_B-1$ or $k_B$ blue edges), while satisfying the other properties. This will be explained at the end of Section \ref{sec:general}.

In the next section, we take the last step to prove Theorem \ref{thm:cycle_choice}.

\subsection{Proof of Theorem \ref{thm:cycle_choice}: Exploiting a crossing}\label{sec:cycle_crossing} 

Recall that out goal is to show the existence of a matching of size $\ell-1$ with $k_R$ red edges and $k_B$ blue edges, which we will refer to as a \emph{desired} matching. 

We first show that, in certain special cases, the proof of Theorem \ref{thm:cycles_no_choice} already gives us a desired matching. 

\begin{obs}
Let $P=\{2v,\dots, 2u-1\}$ be a good path in $G$. If at least one of the edges $2u-1$, $2u$, $2v-2$, $2v+1$ is yellow, then a desired matching exists.
\end{obs}

\begin{proof}
    We refer to the proof of Theorem \ref{thm:cycles_no_choice}, and recall that $M$ is the quasi-matching obtained in the proof. We see that if one of the edges $2u-1, 2u$ is yellow, the matching obtained from $M$ by removing that edge has $k_B$ blue edges and is thus a desired matching.

    Assume now that edges $2u-1, 2u$ are not yellow, and edge $2v+1$ is yellow. Then, thanks to Assumption \ref{ass:no_balanced}, one of the edges $2u-1, 2u$ is red and the other is blue, and edge $2v$ is red or blue. Starting from $M$ and exchanging edge $2v+1$ with edge $2v$, and removing the edge among $2u-1$, $2u$ that has the same color as $2v$, yields a desired matching. We can reason analogously if $2v-2$ is yellow. 
\end{proof}

The above observation allows us to make the following assumption:

\begin{ass}
  \label{ass:good_path} 
    For any good path $P=\{2v,\dots, 2u-1\}$ of $G$, 
     none of the edges $2u-1$, $2u$, $2v-2$, $2v+1$ is yellow.  
\end{ass}


We now resume our investigation of the imbalance curve $d$.
 Due to Assumption \ref{ass:no_balanced}, we can restrict the behavior of $d$: the following can be easily checked via Table \ref{tab:arrows}.
 
\begin{obs}\label{obs:exclude_moves}
We can exclude the following set of two consecutive moves for $d$:
 \[\YB\RY\qquad \BY\YR \qquad \YRRB\qquad \BRRY \qquad  \YB\BR\qquad  \BY\RB\]
\end{obs}


We now consider the curve $d^\infty$, i.e., the extension of $d$ by periodicity defined according to \eqref{eq:f^inf} (with $\tau=\ell$).
We have the following simple observation:

\begin{obs}\label{obs:renumbering}
   Let $\hat G$ be the graph isomorphic to $G$ under the following renumbering of the edges: edge 2 becomes edge 0, edge 3 becomes edge 1, and so on. Let $\hat  d$ be the corresponding imbalance curve. Then $\Im(\hat d^\infty)$ is a translation of $\Im(d^\infty)$. 
\end{obs}
    \begin{proof}
        Using the definition of $d$, it is easy to check that $\hat d^\infty(t)=d^\infty(t+1)-d(1)$ for any $t\in \R$. Due to the periodicity of $d^\infty$ and $\hat d^\infty$, this concludes the proof. 
    \end{proof}

We now prove an important lemma that will allow us to prove Theorem \ref{thm:cycle_choice}. 
Recall that $q\in\R^2$ is defined as $q=(k_R-r_0,k_B-b_0)$.

\begin{lemma}\label{lem:crossing_for_d}
Under Assumptions \ref{ass:open}--\ref{ass:good_path}, and possibly after renumbering the edges of the cycle, we have that:
    \begin{enumerate}
        \item 
        $d^\infty$ is injective.
        \item 
        There exists a crossing pair $(u,v)\in \Z^2$ for $(d^\infty,d+q)$ such that $v < u < v+\ell$.
        \item 
        If the crossing is overlapping, one of the following holds: \begin{enumerate}
            \item there is some $i \in\{1,\dots,v-1\}$ such that  $d^\infty([u-i,u])$ and $d([v-i,v])+q$ form precisely the overlapping part of the crossing;
            \item there is some $i \in \left\{1,\dots,\min\left\{v-1,\ceil{\frac{v-u+\ell}{2}}-1\right\}\right\}$ such that  $d^\infty([u,u+i])$ and $d([v-i,v])+q$ form precisely the overlapping part of the crossing.
        \end{enumerate}
    \end{enumerate}  
\end{lemma}

Note that the injectivity of $d^\infty$ stated in property 1 of the above lemma ensures that the notion of a crossing pair is well defined. Furthermore, recall that, by definition of crossing pair, property 2 also gives the conditions $0<v<\ell$ and $d^\infty(u) = d(v)+q$.

The two subcases of property 3 correspond to overlapping crossings where $d^\infty$ and $d+q$ have (a) the same orientation or (b) opposite orientations in the overlapping part. In case (a), $d^\infty(u-j) = d(v-j)+q$ for all $j\in\{0,\dots, i\}$, while in case (b), $d^\infty(u+j) = d(v-j)+q$ for all $j\in\{0,\dots, i\}$. In both cases, $i$ is maximal with respect to this property. See Figure \ref{fig:overlap} for an illustration of these two cases.




\begin{proof}[Proof of Lemma \ref{lem:crossing_for_d}]
 \begin{enumerate}
        \item Assume by contradiction that $d^\infty$ is not injective:  then, as $d^\infty$ is the extension of $d$ by periodicity, there are distinct $u,v\in [0,\ell)$ with 
    $d(u)-d(v)=k\Delta$ for some $k\in \Z_{\geq 0}$, where we write $\Delta = d(\ell)$ for brevity.  Fix a minimum such $k$. Note that $u,v$ can be assumed to be integer by Observation \ref{cl:d_int}.
    \begin{itemize}
        \item If $k=0$, there are integers $u,v\in\{0,\dots,\ell-1\}$ with $u>v$ and $d(u)=d(v)$, leading to $\delta(\{2v,\dots, 2u-1\})=0$. Thus, there is a balanced path, contradicting Assumption \ref{ass:no_balanced}.

\item If $k\geq 1$, there are integers $u,v\in\{0,\dots,\ell-1\}$ with $u\neq v$ and $d(u)=d(v)+k\Delta$. Suppose first that $k=1$ and $u>v$: then we have $\delta(\{2v,\dots, 2u-1\})=\Delta$; since $\delta(\{0,\dots,2\ell-1\})=\Delta$, we obtain $\delta(\{2u,\dots,2\ell-1, 0,\dots, 2v-1\})=0$, contradicting Assumption \ref{ass:no_balanced}. 

In all other cases, there is a path $P$ such that either $\delta(P)=k\Delta$ (if $u>v$ and $k>1$) or $\delta(P)=-k\Delta$ (if $u<v$), in which case the path $Q=E\setminus P$ satisfies $\delta(Q)=(k+1)\Delta$. We only treat the first case, as, by renaming $k:=k+1$ and $P:=Q$, the second case reduces to the first. 
Let us change the numbering of the cycle so that it starts from $P$: its first edge (which is even) becomes edge 0, its second edge becomes edge 1, and so on. By repeatedly applying Observation \ref{obs:renumbering}, one can see that, if $d$ is updated consequently, the only effect of the renumbering is that $d^\infty$ is translated, and therefore the minimality of $k$ is unaffected. 
Let $\tilde d$ be the restriction of (the updated) $d$ to the edges of $P$, i.e., $\tilde d$ is defined on $[0,\ell']$, with $\ell'=|P|/2$, and satisfies $\tilde d(t)=d(t)$ for every $t\in[0,\ell']$. In particular, $\tilde d(0)=d(0)=(0,0)$ and $\tilde d(\ell')=\delta(P)=k\Delta$. 
        Since $\Delta$ is in the open segment between $\tilde d(0)$ and $\tilde d(\ell')$, we can apply Lemma \ref{lem:curves} to $\tilde d$ and $\Delta$:
 we have that there exist $i\in[0,2\ell']$ and $j\in[0,\ell']$ such that  $j<i<j+\ell'$ and 
    $\tilde d^\infty(i) = \tilde d(j)+\Delta$. Again, $i,j$ can be assumed to be integer. Now, we distinguish two cases again. If $i\leq \ell'$, then $d(i)=\tilde d(i)= \tilde d^\infty(i)=\tilde d(j)+\Delta=d(j)+\Delta$, so the path $\{2j,\dots,2i-1\}$ has imbalance equal to $\Delta$ and therefore its complement is a balanced path, contradicting Assumption \ref{ass:no_balanced}. Otherwise, $i>\ell'$ and  $\tilde d(j)+\Delta=\tilde d^\infty(i) = \tilde d(i-\ell') + \tilde d(\ell')-\tilde d(0)= \tilde d(i-\ell') + k\Delta$,
    which gives $ d(j)- d(i-\ell')= \tilde d(j)-\tilde d(i-\ell')= (k-1)\Delta$,
    contradicting the minimality of $k$.
    \end{itemize}
    
    \item Due to Assumption \ref{ass:Y}, we can fix the numbering of the cycle so that at least one of the edges $\ell-2, 1$ is yellow; since these two edges have different parity, this can be done without changing which edges of $G$ are even and which are odd. Now, we verify that $q\notin\Im(d)$. If $q\in\Im(d)$, then, as $q\in\Z^2$, $q=d(t)$ for some $t\in\{0,\dots,\ell\}$. If $q=d(0)=(0,0)$, we obtain $(k_R,k_B)=(r_0,b_0)=p_{M_0}$, contradicting Assumption \ref{ass:open}. Similarly, if $q=d(\ell)=(r_1-r_0,b_1-b_0)$, we find $(k_R,k_B)=p_{M_1}$, again a contradiction.  So  $t\in \{1,\dots,\ell-1\}$, which implies that $P=\{0,\dots, 2t-1\}$ is a good path, contradicting Assumption \ref{ass:good_path}.  Thus, $q\notin\Im(d)$. 
        
 The existence of a (not necessarily integer) crossing pair now follows from Lemma \ref{lem:crossing}. By 
Observation \ref{cl:d_int}, we see that $u, v\in\Z$.
\item If the crossing is overlapping, there are two cases, depending on whether the orientation of $d^\infty$ and $d+q$ in the overlapping part is the same or not. In the first case, we have $d^\infty([u-i, u]) = d([v-i', v])+q$ for some $i,i'>0$ that are breakpoints of $d$, where we take $i,i'$ maximal. Notice that $i'<v$ by definition of crossing pair. Moreover, due to the structure of $d$, we have $i=i' \in \Z_{>0}$, implying $i\in\{1,\dots,v-1\}$. 
In the second case, we have $d^\infty([u+i, u]) = d([v-i', v])+q$ for some $i,i'>0$, and as above we conclude $i=i'\in \Z_{>0}$ and $i\leq v-1$. The bound on $i$ follows from the fact that if $u+j = v-j+\ell$ for some $j\in \{1,\dots,i\}$, then $d(v-j)+q=d^\infty(u+j)= d(v-j)+d(\ell)$, leading to $q=d(\ell)$, a contradiction to part 1. Hence, we must have $u+i<v-i+\ell$, leading to property (b).
        \end{enumerate}

   
\end{proof}

From now on, we will work with a numbering of the edges of $G$ satisfying the conditions of Lemma \ref{lem:crossing_for_d}.


We now start the actual proof of Theorem \ref{thm:cycle_choice}. We first summarize our approach.
The existence of a crossing pair $(u, v)$, together with Lemma \ref{lem:intersecting_pair_good_path}, implies that there is a good path $P=\{2v,\dots,2u-1\}$ with certain properties on the edges close to its extremes. If the crossing is overlapping, there are multiple good paths with additional properties. We will call $d_\inn$ (resp., $d_\out$) the move of $d^\infty$ right before (resp., after) the overlapping part, and we define $\bar d_\inn$ (resp., $\bar d_\out$) similarly for $d+q$ (see below for the precise definitions).

Then, we will distinguish some cases: 
 \begin{enumerate}
    \item\label{case:simple} The crossing is simple;

    \item\label{case:a)} The crossing is overlapping of type (a) (we refer to Lemma \ref{lem:crossing_for_d}), i.e., $d^\infty$ and $d+q$ have the same orientation in the overlapping part.
 \item\label{case:b)} The crossing is overlapping of type (b), i.e., $d^\infty$ and $d+q$ have opposite orientations in the overlapping part. 
\end{enumerate}
(See Figure \ref{fig:overlap} for an illustration of the last two cases.)

\begin{figure}[h]
    \centering     \centering \begin{subfigure}[h]{0.45\textwidth}
      \begin{tikzpicture}[scale=0.45]

    \draw[step=1, lightgray, very thin] (0,0) grid (7,7);

    \draw[->] (0,0) -- (7.5,0) node[right] {$x_R$};
    \draw[->] (0,0) -- (0,7.5) node[above] {$x_B$};

\node[circle,fill=red, inner sep=1.2] (q) at  (2,1){};
\node[right] at (2,1) {$q$};

\node[circle,fill=black, inner sep=1] (dl) at  (4,2){};
\node[right] at (4,2) {$d(\ell)$};
\node[circle,fill=black, inner sep=1] (d0) at  (0,0){};
\node[left] at (0,0) {$d(0)$};

\node[circle,fill=red, inner sep=1] (q') at  (6,3){};

\draw[very thick] (0,0) -- (0,2) -- (2,2) -- (2,3) -- (3,3) -- (4,2);

\draw[very thick] (2,2) -- (2,3);

\draw[red, thick] (q) -- (2,2);
\draw[red, thick] (3,3) -- (4,3) -- (4,4) --(5,4) -- (6,3);
\draw[red, very thick, dotted] (2,2) -- (2,3) -- (3,3);

\node[circle, fill= black, inner sep=2] (o) at (10.5, 4) {};
\def\vecL{1.5}
\draw[->, thick] ($(o)+(-\vecL,0)$) -- (o); 
\draw[->, thick, red] ($(o)+(0,-\vecL)$) -- (o); 
\draw[->, thick] (o) -- ($(o)+(\vecL-0.1,-\vecL+0.1)$);  
\draw[->, thick, red] (o) -- ($(o)+(\vecL,0)$);   

\node[circle, inner sep=2] (o) at (14.5, 4) {};
\end{tikzpicture}  
\caption{}
    \end{subfigure} \hspace{0.5cm} 
    \begin{subfigure}{0.45\textwidth}
        
      \begin{tikzpicture}[scale=0.45]

    \draw[step=1, lightgray, very thin] (0,0) grid (7,7);

    \draw[->] (0,0) -- (7.5,0) node[right] {$x_R$};
    \draw[->] (0,0) -- (0,7.5) node[above] {$x_B$};

\node[circle,fill=red, inner sep=1.2] (q) at  (2,1){};
\node[right] at (2,1) {$q$};

\node[circle,fill=red, inner sep=1] (q') at  (6,3){};

\node[circle,fill=black, inner sep=1] (dl) at  (4,2){};
\node[right] at (4,2) {$d(\ell)$};
\node[circle,fill=black, inner sep=1] (d0) at  (0,0){};
\node[left] at (0,0) {$d(0)$};

\draw[very thick] (0,0) -- (0,4) -- (2,4);

\draw[very thick] (2,4) -- (2,2);

\draw[very  thick] (2,2) -- (4,2);

\draw[red, thick] (q) -- (2,2);
\draw[red, thick] (2,4) -- (2,5) -- (4,5) -- (4,3) -- (6,3);
\draw[red, very thick, dotted] (2,4) -- (2,2);
\node[circle, fill= black, inner sep=2] (o) at (10.5, 4) {};
\def\vecL{1.5}
\draw[->, thick] ($(o)+(-\vecL,0)$) -- (o); 
\draw[->, thick, red] ($(o)+(0,-\vecL)$) -- (o); 
\draw[->, thick] (o) -- ($(o)+(\vecL,0)$);  
\draw[->, thick, red] (o) -- ($(o)+(0,\vecL)$);   

\node[circle, inner sep=2] (o) at (14.5, 4) {};
\end{tikzpicture}
   \caption{}
    \end{subfigure}
    \caption{
    Examples of an overlapping crossing of type (a) (left) and of type (b) (right), where we only show $d$ (black) and $d+q$ (red). In both cases, the crossing pair is $(u,v)=(6,3)$ and the \lq \lq length'' of the overlapping part is $i=2$. At the side of each plot we draw (translates of) the vectors $d_\inn$, $\bar d_\out$, $d_\out$, $\bar d_\inn$, in clockwise order, with $d_\inn= (1,0)$ in both cases.
    }
    \label{fig:overlap}
\end{figure}
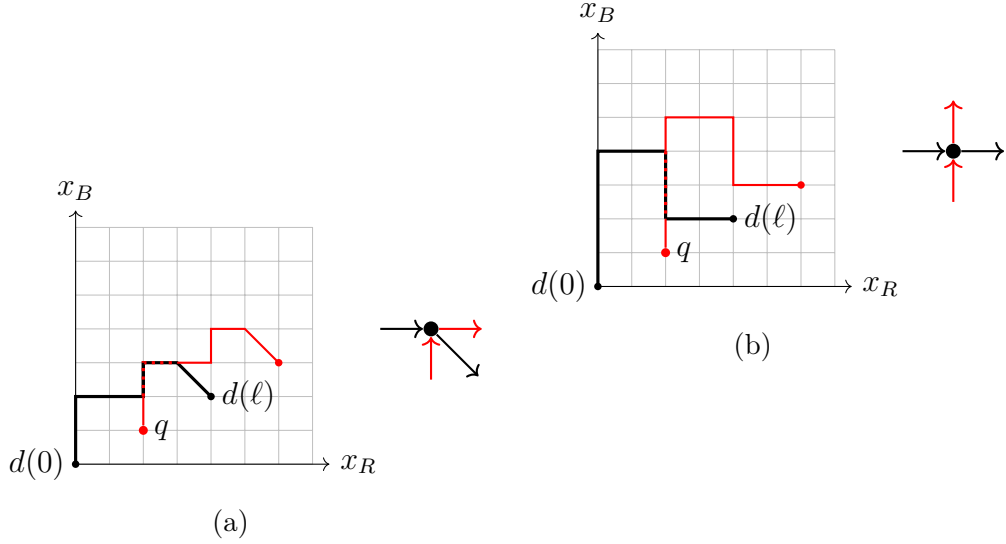

 We will conclude the proof of Theorem \ref{thm:cycle_choice} by arguing that in each case there is a contradiction.

\paragraph{Case \ref{case:simple}: The crossing is simple}
  

 We define \[d_{\inn} := d^\infty(u)-d^\infty(u-1),\quad  
    \bar d_{\inn} := d(v)-d(v-1),
    \]
     \[
d_{\out} := d^\infty(u+1)-d^\infty(u), \quad \bar d_{\out} :=d(v+1)-d(v).
\]
In this case, the four vectors $-d_{\inn}$, $d_{\out}$, $-\bar d_{\inn}$, $\bar d_{\out}$ are all distinct (where we changed the sign of the ``in'' vectors so that all vectors point away from the crossing point $d^\infty(u)=d(v)+q$), and \lq\lq cross'' in the sense that they alternate between vectors $d_{*}$ and $\bar d_{*}$ in, say, clockwise order. 
We will rule out all the cases in which such crossings can occur. Recall that $P=\{2v,\dots, 2u-1\}$ is a good path, by Lemma \ref{lem:intersecting_pair_good_path}. By equation \eqref{eq:d-delta}, and using Assumptions \ref{ass:no_balanced} and \ref{ass:good_path},  
 the possible moves for $d_\inn$ and $d_\out$ are those listed in Table \ref{tab:moves_inout}.
 
 \begin{table}[h]
    \centering
    \begin{subtable}{0.45\textwidth}
        \centering
        \begin{tabular}{c|c|c}
            $2u-2$ & $2u-1$ & \\\hline
            $R$ & $B$ & \RB \\
            $B$ & $R$ & \BR \\
            $Y$ & $B$ & \YB \\
            $Y$ & $R$ & $\rightarrow$
        \end{tabular}
        \caption*{$d_\inn$}
        \label{tab:din}
    \end{subtable}
    \hfill
    \begin{subtable}{0.45\textwidth}
        \centering
        \begin{tabular}{c|c|c}
            $2u$ & $2u+1$ & \\\hline
            $R$ & $B$ & \RB \\
            $B$ & $R$ & \BR \\
            $R$ & $Y$ & $\leftarrow$ \\
            $B$ & $Y$ & \BY
        \end{tabular}
        \caption*{$d_\out$}
        \label{tab:dout}
    \end{subtable}
    \caption{Possible moves for $d_\inn$ and $d_\out$ in Case \ref{case:simple}.}
    \label{tab:moves_inout}
\end{table}

Note also that, again by Assumptions \ref{ass:no_balanced} and \ref{ass:good_path}, $\bar d_\inn$ has the same possible moves as $d_\out$, and $\bar d_\out$  has the same possible moves as $d_\inn$. Moreover, not all combinations for $d_\inn$, $d_\out$ are allowed, due to Observation \ref{obs:exclude_moves}. Thus, we obtain the following cases, which we depict graphically: the black dot represents the crossing point. 

For $(d_\inn, d_\out)$, the possible moves are:
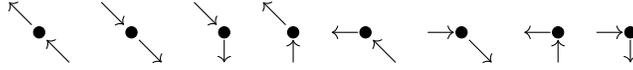
\[
 \tikz[baseline]{\draw[->] (0.8,0) -- (0.5,0.3); \node[pointv] at (0.4, 0.4) {}; \draw[->] (0.3,0.5) -- (0,0.8);}  \quad
 \tikz[baseline]{\draw[<-] (0.8,0) -- (0.5,0.3); \node[pointv] at (0.4, 0.4) {}; \draw[<-] (0.3,0.5) -- (0,0.80);} \quad
  \tikz[baseline]{\draw[<-] (0.4,0) -- (0.4,0.3); \node[pointv] at (0.4, 0.4) {}; \draw[<-] (0.3,0.5) -- (0,0.80);} \quad
 \tikz[baseline]{\draw[->] (0.4,0) -- (0.4,0.3); \node[pointv] at (0.4, 0.4) {}; \draw[->] (0.3,0.5) -- (0,0.80);} \quad
 \tikz[baseline]{\draw[->] (0.8,0) -- (0.5,0.3); \node[pointv] at (0.4, 0.4) {}; \draw[->] (0.3,0.4) -- (-0.05,0.4);} \quad
 \tikz[baseline]{\draw[<-] (0.8,0) -- (0.5,0.3); \node[pointv] at (0.4, 0.4) {}; \draw[<-] (0.3,0.4) -- (-0.05,0.4);} \quad
 \tikz[baseline]{\draw[->] (0.3,0.4) -- (-0.05,0.4); \node[pointv] at (0.4, 0.4) {}; \draw[->] (0.4,0) -- (0.4,0.3);} \quad
 \tikz[baseline]{\draw[<-] (0.3,0.4) -- (-0.05,0.4); \node[pointv] at (0.4, 0.4) {}; \draw[<-] (0.4,0) -- (0.4,0.3);}
\]
Notice that the last two possibilities can be ruled out because they cannot lead to a simple crossing.
For $(\bar d_\inn, \bar d_\out)$, the possible moves are:
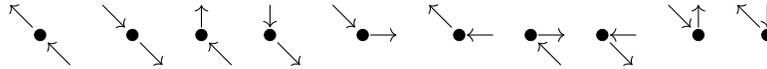
\[
 \tikz[baseline]{\draw[->] (0.8,0) -- (0.5,0.3); \node[pointv] at (0.4, 0.4) {}; \draw[->] (0.3,0.5) -- (0,0.8);}  \quad
 \tikz[baseline]{\draw[<-] (0.8,0) -- (0.5,0.3); \node[pointv] at (0.4, 0.4) {}; \draw[<-] (0.3,0.5) -- (0,0.80);} \quad 
 \tikz[baseline]{\draw[->] (0.8,0) -- (0.5,0.3); \node[pointv] at (0.4, 0.4) {}; \draw[->] (0.4,0.5) -- (0.4,0.80);}\quad
 \tikz[baseline]{\draw[<-] (0.8,0) -- (0.5,0.3); \node[pointv] at (0.4, 0.4) {}; \draw[<-] (0.4,0.5) -- (0.4,0.80);}\quad
\tikz[baseline]{\draw[<-] (0.85,0.4) -- (0.5,0.4); \node[pointv] at (0.4, 0.4) {}; \draw[<-] (0.3,0.5) -- (0,0.8);}\quad 
\tikz[baseline]{\draw[->] (0.85,0.4) -- (0.5,0.4); \node[pointv] at (0.4, 0.4) {}; \draw[->] (0.3,0.5) -- (0,0.8);}\quad 
\tikz[baseline]{\draw[<-] (0.85,0.4) -- (0.5,0.4); \node[pointv] at (0.4, 0.4) {}; \draw[<-] (0.5,0.3) -- (0.80,0);}\quad 
\tikz[baseline]{\draw[->] (0.85,0.4) -- (0.5,0.4); \node[pointv] at (0.4, 0.4) {}; \draw[->] (0.5,0.3) -- (0.80,0);}\quad 
\tikz[baseline]{\draw[->] (0.4,0.5) -- (0.4,0.8); \node[pointv] at (0.4, 0.4) {}; \draw[<-] (0.3,0.5) -- (0,0.8);}\quad 
\tikz[baseline]{\draw[<-] (0.4,0.5) -- (0.4,0.8); \node[pointv] at (0.4, 0.4) {}; \draw[->] (0.3,0.5) -- (0,0.8);}
\]
As above, the last four possibilities can be ruled out because they cannot lead to a simple crossing.
Finally, it is easy to check that none of the remaining combinations of moves for $(d_\inn, d_\out)$ and $(\bar d_\inn, \bar d_\out)$ leads to a simple crossing (either because two segments coincide or because it is not a crossing at all). This concludes the proof in this case.

\paragraph{Case \ref{case:a)}: The crossing is overlapping of type (a)}

 By Lemma \ref{lem:crossing_for_d}, $d^\infty([u-i,u])$ and $d([v-i,v])+q$ form precisely the overlapping part of the crossing, i.e., there is some $i\in \{1,\dots,v-1\}$ such that $d^\infty(u-j) = d(v-j)+q$ for all $j\in\{0,\dots, i\}$, where $i$ maximal with respect to this property. 
 Also recall that, due to Lemma \ref{lem:intersecting_pair_good_path}, $P_j = \{2(v-j),\dots, 2(u-j)-1\}$ is a good path for all $j\in\{0,\dots,i\}$. 
  
By subtracting equation $d^\infty(u-j-1) = d(v-j-1)+q$ from $d^\infty(u-j) = d(v-j)+q$ for each $j\in\{0,\dots, i-1\}$, we obtain
\[d^\infty(u-j)-d^\infty(u-j-1)= d(v-j)-d(v-j-1)\] for $j\in\{0,\dots, i-1\}$.
Since moves of $d$ are in one-to-one correspondence with pairs of colors via equation \eqref{eq:d-delta}, this implies that the edges $2(v-j)-1$ and $2(u-j)-1$ have the same color, and so do the edges $2(v-j)-2$ and $2(u-j)-2$, for each $j\in\{0,\dots, i-1\}$. Moreover, by Assumption \ref{ass:good_path}, for every $j\in\{0,\dots, i\}$ edges $2(v-j)-2$, $2(v-j)+1$, $2(u-j)-1$, $2(u-j)$ are not yellow. This, by Assumption \ref{ass:no_balanced}, implies that the edges from $2(v-i)$ to $2v-1$ form an alternating sequence of red and blue; say that the even edges are red and the odd edges are blue, the other case being analogous. Then, the same holds for the edges from $2(u-i)-1$ to $2u$, where, for the first and the last edge, we again used Assumption \ref{ass:no_balanced} along with the fact that they are not yellow. Hence, the overlapping part is diagonal and consists of moves of the type $\RB$, and also each of the edges $2(v-i)-2$, $2v+1$ is red or blue.   
     
     Now, we define \[d_{\inn} := d^\infty(u-i)-d^\infty(u-i-1),\quad  
    \bar d_{\inn} := d(v-i)-d(v-i-1),
    \]
    \[
d_{\out} := d^\infty(u+1)-d^\infty(u), \quad \bar d_{\out} :=d(v+1)-d(v).
\]

By \eqref{eq:d-delta}, $\bar d_\inn$ and $\bar d_\out$ depend on the color of the edges $2(v-i)-2,2(v-i)-1$ and $2v,2v+1$, respectively.
As each of the edges $2(v-i)-2$, $2v+1$ is red or blue, the possible cases for $\bar d_\inn$, $\bar d_\out$ are:
     \[
     \bar d_\inn \in \{\RB, \leftarrow, \BY\}, \quad \bar d_\out \in \{\RB, \YB, \rightarrow\}.
     \]
     (The move $\BR$ is ruled out in both cases by the injectivity of $d^\infty$.)
     Similarly, as edge $2(u-i)-1$ is blue while edge $2u$ is red, the possible cases for $d_\inn$, $d_\out$ are:
     \[
    d_\inn \in \{\RB, \YB\}, \quad d_\out \in \{\RB, \leftarrow\}.
     \]
However, it is easy to check that in any combination of moves, $\bar d_\inn$ and $\bar d_\out$ lie in the same one of the two regions determined by $d^\infty$, contradicting the fact that $(u,v)$ is a crossing pair.
(See Figure \ref{fig:cases2-3}, left.)

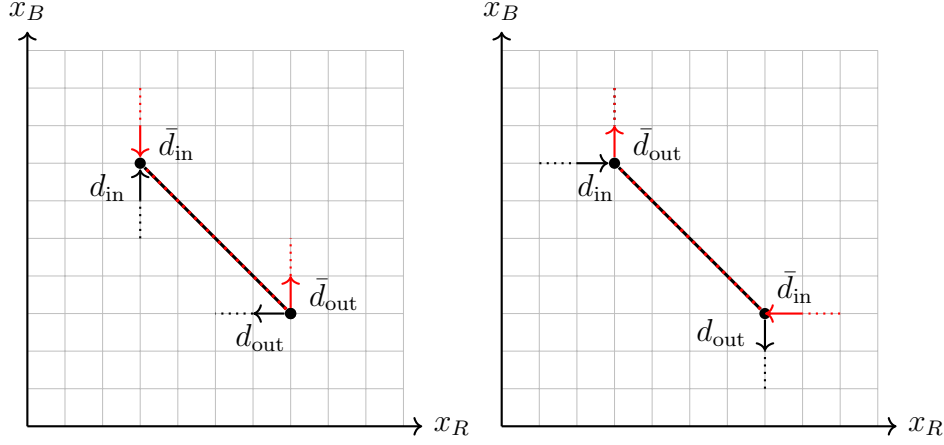
\begin{figure}[h]
    \centering
    \begin{subfigure}{0.45\textwidth}
      \begin{tikzpicture}[scale=0.5]

    \draw[step=1, lightgray, very thin] (0,0) grid (10,10);

    \draw[thick,->] (0,0) -- (10.5,0) node[right] {$x_R$};
    \draw[thick,->] (0,0) -- (0,10.5) node[above] {$x_B$};
    
  \begin{scope}[shift={(1,-1)}]

\node[circle,fill=black, inner sep=1.5] (t0) at  (2,8){};

\node[circle,fill=black, inner sep=1.5] (t0mi) at  (6,4){};
    \node (mid) at (4,6) {};

\draw[very thick] (t0mi) -- (t0);
\draw[very thick, red, dotted] (t0mi) -- (t0);

\draw[thick, dotted, red] (2,10) -- (2,9);
\draw[  <-, red, thick] (t0) -- (2,9);
\node[right] at (2.2,8.5) {$\bar d_\inn$};

\draw[thick, dotted] (5,4) -- (4,4);
\draw[ thick, <-] (5,4) -- (t0mi);
\node[below] at (5.2,4) 
{$d_\out$};

\draw[thick, dotted] (2,6) -- (2,7);
\draw[thick, ->] (2,7) -- (t0);
\node[left] at (1.9,7.35) {$d_\inn$};

\draw[ thick, red, ->] (6,4) -- (6,5);
\node[right] at (6.2,4.5) {$ \bar d_\out$};
\draw[thick,red ,  dotted] (6,5) -- (6,6);

\node[circle,fill=black, inner sep=1.5] (t0mi) at  (6,4){};

\end{scope}

\end{tikzpicture}
 \end{subfigure}
 \begin{subfigure}{0.45\textwidth}
    \begin{tikzpicture}[scale=0.5]

    \draw[step=1, lightgray, very thin] (0,0) grid (10,10);

    \draw[thick,->] (0,0) -- (10.5,0) node[right] {$x_R$};
    \draw[thick,->] (0,0) -- (0,10.5) node[above] {$x_B$};

  \begin{scope}[shift={(1,-1)}]
\node[circle,fill=black, inner sep=1.5] (t0) at  (2,8){};

\draw[thick, dotted] (2,10) -- (2,9);

\node[circle,fill=black, inner sep=1.5] (t0mi) at  (6,4){};
    \node (mid) at (4,6) {};

\draw[very thick] (t0mi) -- (t0);
\draw[very thick, red, dotted] (t0mi) -- (t0);

\draw[thick,red,->] (t0) -- (2,9);
\node[right] at (2.2,8.5) {$\bar d_\out$};

\draw[thick, dotted] (6,2) -- (6,3);
\draw[thick,->] (t0mi) -- (6,3);
\node[left] at (5.8,3.5) {$ d_\out$};

\draw[thick,red, dotted] (2,10) -- (2,9);
\draw[thick,->] (1,8) -- (t0);
\node[below] at (1.5,8) {$d_\inn$};

\draw[thick,red, ->] (7,4) -- (6,4);
\node[above] at (6.8,4) {$\bar d_\inn$};

\draw[thick, dotted] (0,8) -- (1,8);
\draw[thick, red, dotted] (8,4) -- (7,4);
\end{scope}
\end{tikzpicture}
 \end{subfigure}
    \caption{
    An example of the contradictions obtained in case \ref{case:a)} (left) and in case \ref{case:b)} (right). The possible moves for vectors $d_\inn$, $d_\out$, $\bar d_\inn$, $\bar d_\out$ conflict with $(u,v)$ being a crossing pair for $(d^\infty,d+q)$.}
    \label{fig:cases2-3}
\end{figure}

\paragraph{Case \ref{case:b)}: The crossing is overlapping of type (b)}

By Lemma \ref{lem:crossing_for_d},  $d^\infty([u,u+i])$ and $d([v-i,v])+q$ form precisely the overlapping part of the crossing, i.e., there is some integer $i \in \left\{1,\dots,\min\left\{v-1,\ceil{\frac{v-u+\ell}{2}}-1\right\}\right\}$ such that $d^\infty(u+j) = d(v-j)+q$ for all $j\in\{0,\dots, i\}$, where $i$ is maximal with respect to this property. The condition $v<u$ and the bound $i\le\ceil{\frac{v-u+\ell}{2}}-1$ imply $v-j<u+j<v-j+\ell$ for all  $j\in\{0\dots,i\}$, so, by Lemma \ref{lem:intersecting_pair_good_path}, path $P_j = \{2(v-j),\dots, 2(u+j)-1\}$ is good. 
   
By subtracting equation $d^\infty(u+j-1) = d(v-j+1)+q$ from $d^\infty(u+j) = d(v-j)+q$ for each $j\in\{1,\dots, i\}$, we obtain
\[d^\infty(u+j)-d^\infty(u+j-1)= d(v-j)-d(v-j+1)= -(d(v-j+1)-d(v-j))\] for $j\in\{1,\dots, i\}$.
     Since moves of $d$ are in one-to-one correspondence with pairs of colors via \eqref{eq:d-delta}, and opposite moves of $d$ correspond to swapping colors, edges $2(v-j)+1$ and $2(u+j)-2$  have the same color, and so do edges
     $2(v-j)$ and $2(u+j)-1$, for every $j\in\{1,\dots, i\}$.
     Moreover, by Assumption \ref{ass:good_path}, for each $j\in\{0,\dots, i\}$ none of the edges $2(v-j)-2$, $2(v-j)+1$, $2(u+j)-1$, $2(u+j)$ is yellow. This, by Assumption \ref{ass:no_balanced}, implies that the edges from $2(v-i)$ to $2v-1$ form an alternating sequence of red and blue; say that the even edges have color red and the odd edges have color blue. Then, the opposite holds for the edges from $2u-1$ to $2(u+i)$ (i.e., the odd edges are red and the even are blue), where for the first and the last edge we again used Assumption \ref{ass:no_balanced} and the fact that they are not yellow. Hence the overlapping segment is diagonal, and also each of the edges $2(v-i)-2$, $2v+1$ is red or blue.

      Now, we define \[ d_{\out} := d^\infty(u+i+1)-d^\infty(u+i),\quad  
    \bar d_{\inn} := d(v-i)-d(v-i-1),
    \]
    \[
d_{\inn} := d^\infty(u)-d^\infty(u-1), \quad \bar d_{\out} :=d(v+1)-d(v).
\]
As each of the edges $2(v-i)-2$, $2v+1$ is red or blue, the possible cases for $\bar d_\inn$, $\bar d_\out$ are:
     \[
     \bar d_\inn \in \{\RB, \leftarrow, \BY\}, \quad \bar d_\out \in \{\RB, \YB, \rightarrow\}.
     \]
     The possible cases for $d_\inn$, $d_\out$ are:
     \[
    d_\inn \in \{\BR, \rightarrow\}, \quad d_\out \in \{\BR, \BY\},
     \]
     where we used the fact that edge $2(u+i)$ is blue and edge $2u-1$ is red.
However, it is easy to check that in any combination of moves, $\bar d_\inn$ and $\bar d_\out$ lie in the same one of the two regions determined by $d^\infty$, contradicting the fact that $(u, v)$ is a crossing pair. (See Figure \ref{fig:cases2-3}, right.) This settles the last case and completes the proof of Theorem \ref{thm:cycle_choice}.

\section{Main result and proof}\label{sec:general}

In this section we prove our main result, Theorem \ref{thm:main}. 
The key ingredients are two extensions of Theorem \ref{thm:cycle_paths_choice}.
In the first, we extend Theorem \ref{thm:cycle_paths_choice}  to the case in which the point $(k_R, k_B)$ has only one integer component, which we take to be the first one. 

\begin{lemma}[Fractional version]\label{lem:cycle_frac}
    Let $G$ be a colored path or even cycle, and denote by $M_0, M_1$ the matchings of even and odd edges, respectively. Assume that $(k_R, k_B)$ is on the segment between $p_{M_0}, p_{M_1}$, with $k_R$ integer. Then there is a matching $M$ with $|M|\geq |M_1|-1$, exactly $k_R$ red edges, and $\lceil k_B\rceil$ or $\lceil k_B\rceil-1$ blue edges.  This matching can be found efficiently. 
\end{lemma}
\begin{proof} By reasoning exactly as in the proof of Theorem \ref{thm:cycle_paths_choice},  we can assume that $G$ is an even cycle. Clearly, we can also restrict ourselves to the case where $k_B$ is not integer, as otherwise we can just apply Theorem \ref{thm:cycle_choice}.
This implies, in particular, that $(k_R,k_B)$ is in the {\em open} segment between $p_{M_0},p_{M_1}$.
We use the notation $p_{M_0}=(r_0, b_0)$ and $p_{M_1}=(r_1, b_1)$.

Consider the imbalance curve $d$ defined in Section \ref{sec:cycle}, and let $q=(k_R-r_0, k_B-b_0)$ and $\lceil q \rceil = (k_R-r_0, \lceil k_B \rceil-b_0)$. Now, if there exists a path $P$ that starts at an even edge and satisfies $\delta(P) = \lceil q \rceil$, we can argue exactly as in the proof of Theorem \ref{thm:cycles_no_choice} to find a desired matching, and we are done. 

To show the existence of such a path $P$, it suffices to find an intersecting pair $(\bar u,\bar v)$ for C such that $\bar v<\bar u<\bar v+\ell$, and then argue as in Lemma \ref{lem:intersecting_pair_good_path}.
To find such an intersecting pair, we will start from the existence of an intersecting pair $(u,v)$ for $(d^\infty,d+q)$, which is guaranteed by Lemma \ref{lem:curves}, and we will show that, by rounding $u, v$ up or down as needed, we can obtain an intersecting pair for $(d^\infty,d+\ceil q)$ (see Figure \ref{fig:fractional}). This plan, which can be easily checked by hand, requires however a fair amount of details. 
Consider the curve $d+q$: the images of its breakpoints  are of the form $(x,y+\beta)$ for $x,y\in \Z$, where $\beta=k_B-\floor{k_B}$. Hence, each of the segments $(d+q)|_{[t,t+1]}$ with $t\in \{0,\dots,\ell-1\}$ can be horizontal (with extremes $(x,y+\beta)$, $(x+1,y+\beta)$), vertical (with extremes $(x,y+\beta)$, $(x,y+1+\beta)$), or diagonal (with extremes $(x,y+\beta)$, $(x+1, y-1+\beta)$), where $x,y\in \Z$. Clearly, the curve $d+ \lceil q \rceil$ is obtained by translating $d+q$ by $(0,1-\beta)$.
Moreover, since $q=(k_R,k_B)-p_{M_0}$ and $(k_R,k_B)$ is in the open segment between $p_{M_0}$ and $p_{M_1}$, the point $q$ is in the open segment between $d(0)=(0,0)$ and $d(\ell)=p_{M_1}-p_{M_0}$. Lemma \ref{lem:curves} then guarantees the existence of an intersecting pair $(u,v)$ for $(d^\infty,d+q)$ such that $v<u<v+\ell$; thus, $d^\infty(u)=d(v)+q$. Now, using the fact that the possible moves of $d$ are only those listed in Table \ref{tab:arrows}, it is easy to verify that there is also an intersecting pair $(\bar u,\bar v)$ for  $(d^\infty,d+\lceil q \rceil)$ obtained by suitably choosing $\bar u\in\{\lfloor u \rfloor,\lceil u \rceil\}$ and $\bar v\in\{\lfloor v \rfloor,\lceil v \rceil\}$. Therefore, if this intersecting pair satisfies $\bar v < \bar u <\bar v+\ell$, we are done.

To conclude, we need to deal with the cases in which condition $\bar v < \bar u <\bar v+\ell$ is violated. Assume first that $\bar v \geq \bar u$. Then, as $v<u$, we must have $\bar v\in\{\bar u,\bar u+1\}$. If $\bar v=\bar u$, from $d^\infty(\bar u)=d(\bar v)+\ceil q$ (which holds because $(\bar u,\bar v)$ is an intersecting pair for  $(d^\infty,d+\lceil q \rceil)$) we obtain $\lceil q \rceil=(0,0)$, which implies $k_R=r_0$, $\lceil k_B \rceil=b_0$. Thus, $M_0$ is a desired matching in this case. If $\bar v=\bar u+1$, then we necessarily have  $\bar u=\floor u$, $\bar v=\ceil v$, and $u,v$ have the same integer part. But then, the intersection point $d^\infty(u)=d(v)+q$ belongs to identical moves of $d^\infty$ and $d+q$. Since the first component of $q$ is integer and the second is fractional, this cannot happen if these identical moves are both horizontal or both diagonal, and it follows that they are both vertical. Furthermore, the conditions $\bar u=\floor u$, $\bar v=\ceil v$ imply that these are moves of the type $\downarrow$. Then, $\ceil q=d^\infty(\bar u)-d(\bar v)=d^\infty(\bar u)-d^\infty(\bar u+1)=(0,1)$, which implies $k_R=r_0$, $\ceil {k_B}=b_0+1$. Thus, again, $M_0$ is a desired matching.

Now assume that $\bar u\geq \bar v+\ell$. Then, as $u<v+\ell$, we must have $\bar u\in\{\bar v+\ell,\bar v+\ell+1\}$. If $\bar u=\bar v+\ell$, then $\lceil q \rceil=d^\infty(\bar u)-d(\bar v)=d(\ell)=(r_1-r_0,b_1-b_0)$, hence $k_R=r_1$, $\lceil k_B \rceil=b_1$, and $M_1$ is a desired matching. If $\bar u=\bar v+\ell+1$, then we necessarily have $\bar u=\ceil u$, $\bar v=\floor v$, and $u,v+\ell$ have the same integer part. As above (using also the periodicity of $d^\infty$), this implies that the intersection point $d^\infty(u)=d(v)+q$ belongs to identical moves of $d^\infty$ and $d+q$, which must be vertical. Furthermore, the conditions $\bar u=\ceil u$, $\bar v=\floor v$ imply that these are moves of the type $\uparrow$. Then, $\ceil q=d^\infty(\bar u)-d(\bar v)=d^\infty(\bar v+\ell+1)-d^\infty(\bar v)=d(\ell)+(0,1)=(r_1-r_0,b_1-b_0+1)$. This implies $k_R=r_1$, $\ceil {k_B}=b_1+1$, and thus $M_1$ is a desired matching.

Finally, just as in Theorem \ref{thm:cycle_paths_choice}, a desired matching can be found efficiently by inspection.
\end{proof}

 \begin{figure}[h]
    \centering     \centering \begin{tikzpicture}[scale=1.1]

    \draw[step=1, lightgray, very thin] (0,0) grid (4,3);

    \draw[->] (0,0) -- (4,0) node[right] {};
    \draw[->] (0,0) -- (0,3) node[above] {};

\draw[very thick] (0,0) -- (0,2) -- (3,2);

\node[circle,fill=black, inner sep=1] (a) at  (0,0){};
\node[circle,fill=black, inner sep=1] (b) at  (3,2){};
\node[circle,fill=gray, inner sep=1] (b) at  (4,3){};

\node[circle,fill=red, inner sep=1.2] (q) at  (1,0.66){};
\node[left] at (q) {$q$};

\draw[red, thick] (q) -- (1,2.66) -- (4,2.66);

\node[circle,fill=yellow, inner sep=1.2] (qc) at  (1,1){};
\draw[gray, very thick, dotted] (qc) -- (1,3) -- (4,3);
\node[right] at (qc) {$\lceil q \rceil$};
\newcommand{\radius}{0.1}
\draw (1-\radius,2-\radius) rectangle (1+\radius,2+\radius);

\end{tikzpicture}
    \caption{Example of the proof of Lemma \ref{lem:cycle_frac}, with $d$ shown in black,  $d+q$ in red, $\lceil q \rceil$ in yellow and $d+\lceil q \rceil$ in gray, and the crossing indicated by a square. (Notice that $q$ lies in the segment between $d(0)$ and $d(5)$.) In this example, $u= 3$, $v=4/3$, and by setting $\bar u = 3$, $\bar v = 1$ we obtain a crossing pair for $(d^\infty, d+\lceil q \rceil)$. 
    }
    \label{fig:fractional}
\end{figure}
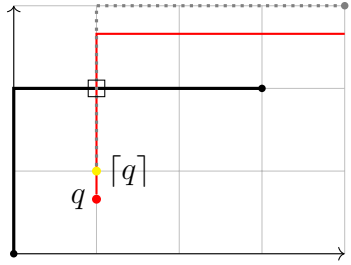

We now need to extend Theorem \ref{thm:cycle_paths_choice} in a different direction: in the hypotheses of that theorem, the graph $G$ is the union of two (disjoint) matchings whose corresponding incidence vectors are adjacent vertices in the matching polytope. We now examine the case where $G$ is the union of any two matchings.

\begin{lemma}[Two general matchings]\label{lem:general}
    Let $G$ be a colored graph that is the union of any two matchings $M_0$, $M_1$, with $|M_0|\geq |M_1|$. Assume that $(k_R, k_B)$ is an integer point on the segment between $p_{M_0}, p_{M_1}$. Then there is a matching $M$ with $|M|\geq |M_1|-2$, exactly $k_R$ red edges, and
$k_B$ or $k_B-1$ blue edges. If $G$ contains no cycles, the bound on the cardinality can be improved to $|M|\ge |M_1|-1$. In all cases, $M$ can be found efficiently. 
\end{lemma}

\begin{proof} 
First, for any edge $e$ that is in both $M_0$ and $M_1$, we remove $e$ from $M_0$ and $M_1$, we include it in the final matching $M$, and modify $(k_R, k_B)$ according to its color (e.g., we decrease $k_R$ by 1 if $e$ is red). This preserves the property that $(k_R, k_B)$ is on the segment between $p_{M_0}, p_{M_1}$. Hence, we can assume that $M_0$ and $M_1$ are disjoint. 
Furthermore, we know that the connected components of $G$ are paths and even cycles in which the edges alternate between $M_0$ and $M_1$. Since for a single connected component the result follows from Theorem \ref{thm:cycle_paths_choice}, in the following we assume that $G$ has at least two components.

We can also assume that in any component $D$ of $G$ there are no consecutive edges with the same color, i.e., the coloring is proper. Indeed, assume that there is a component $D$ that contains two consecutive edges of the same color, say blue. Consider the instance resulting after contracting these two edges to a single node.\footnote{Contracting may produce a trivial component or two parallel edges, but this does not affect the correctness of the argument.} 
 A matching in the contracted instance yields a matching in the original instance with one more blue edge, hence our desired matching exists if a similar matching exists in the contracted instance, where $k_B$ is reduced by 1. Therefore, from now on we assume that the coloring is proper. We now distinguish two cases.

\begin{enumerate}
\item
Assume that $|M_0|>|M_1|$ or $G$ contains some yellow edges. In this case, the idea is to glue the components of $G$ into a single cycle (similarly to what is done in \cite{grandoni2010approximation}) and then apply Theorem \ref{thm:cycle_choice}. We proceed as follows:
\begin{itemize}
\item First, we \lq\lq open up'' each cycle, turning it into an even path;
\item Then, while there are both an $M_0$-augmenting path and an $M_1$-augmenting path, we patch them together (by identifying an endpoint of each) into an even path;
\item Now, we are left with $|M_0|-|M_1|$ paths that are $M_1$-augmenting, and we attach a dummy edge colored yellow to each of them;
\item Finally, we identify the endpoints of the paths obtained in this way (including the original even paths) to create a cycle $\mathcal{C}$.
\end{itemize}
It is easy to check that this can be done so that in $\mathcal{C}$ the edges in $M_0$ are even, and the edges in $M_1$ along with the dummy edges are odd. Notice that $\mathcal{C}$ has $2|M_0|$ edges, $|M_0|-|M_1|$ of which are dummy, and the hypotheses of Theorem \ref{thm:cycle_choice} hold for $\mathcal{C}$ and $(k_R, k_B)$.

Since $|M_0|>|M_1|$ or $G$ contains some yellow edges, $\mathcal{C}$ has at least one yellow edge. Then, from Theorem \ref{thm:cycle_choice}, we obtain a matching $M$ of $\cal C$ of size $|\mathcal{C}|/2-1=|M_0|-1$, with $k_R$ red edges and $k_B$ blue edges. 
Now, since $|M|=|\mathcal{C}|/2-1$, $M$ leaves exposed exactly two nodes $u_1, u_2$ in $\C$; denote by $D_1$ (resp., $D_2$) a component of $G$ containing $u_1$ (resp., $u_2$).\footnote{When $u_1$ (resp., $u_2$) is a node of $\mathcal C$ obtained by identifying two nodes of $G$ belonging to two distinct components, we can choose $D_1$ (resp., $D_2$) to be either component.} Then all the edges of $M$ that in $\mathcal C$ appear on the same path connecting $u_1,u_2$ have the same parity. Thus, after removing the dummy edges, the only thing that could prevent $M$ from being a matching in $G$ is if $M$ contains two consecutive edges from $D_1$ or $D_2$. 
This is possible only if, for some $i\in\{1,2\}$, $D_i$ is an (even) cycle and the two consecutive edges of $D_i$ belonging to $M$ are the first and the last edge of $D_i$ (seen as a path in $\mathcal C$). However, this cannot happen for \emph{both} $D_1$ and $D_2$, as there is an even number of edges separating the last  edge of $D_1$ and the first edge of $D_2$ in $\mathcal C$, and on this path $M$ takes edges with the same parity. Hence, by removing the dummy edges and at most one more edge from $M$ (which we can choose not to be red, since the coloring is proper), we obtain a matching satisfying the requirements. In particular, the cardinality of this matching is at least $|M_1|-2$, which improves to $|M_1|-1$ when no component is a cycle, because in this case only the dummy edges must be removed.

\item
Suppose now that $|M_0|=|M_1|$ and $G$ does not contain any yellow edges.
Hence, $G$ consists of even cycles and paths in which the edges alternate between red and blue, because the coloring is proper. Moreover, we can assume that all the paths are even: indeed, if there is, say, an $M_0$-augmenting path, then since $|M_0|=|M_1|$ there must be an $M_1$-augmenting path and we can join them together, obtaining a new instance with an even alternating path. If this creates two consecutive edges of the same color, we can contract them to a single node as argued above. Again as above, a matching in the new instance satisfying the properties in the thesis yields a matching in the original instance with the same properties.

Note that $p_{M_0}$ and $p_{M_1}$ (hence $(k_R,k_B)$ as well) lie on the line of equation $x_R+x_B=\ell$, with $\ell=|M_0|=|M_1|$. Using the notation $p_{M_0}=(r_0,b_0)$ and $p_{M_1}=(r_1,b_1)$, this reads $r_0+b_0=r_1+b_1=\ell$.
 Consider any component $D$ of $G$, and say that $D$ has $2h^D$ edges. We would like to select the even or the odd edges of $D$ to be part of our desired matching and recurse on $G\setminus D$. Since when there is a single component one can invoke Theorem \ref{thm:cycle_paths_choice}, this will give a matching with at least $|M_1|-1$ edges satisfying the color requirements.
 
Let $r_0^D$ (resp., $b_0^D$) be the number of red (resp., blue) edges of $M_0\cap D$, and define $r_1^D, b_1^D$ analogously. We have $r_0^D+b_0^D= r_1^D+b_1^D=h^D$.
 The points
 \begin{align*}
 p^0=(r_0-r_0^D,b_0-b_0^D),&\quad k^0= (k_R-r_0^D, k_B-b_0^D),\\
 p^1=(r_1-r_1^D,b_1-b_1^D),&\quad k^1= (k_R-r_1^D, k_B-b_1^D)
 \end{align*}
 all lie on the line of equation $x_R+x_B=\ell-h^D$: $p^0, p^1$ represent the matchings $M_0$, $M_1$ restricted to $G\setminus D$, while $k^0$ (resp., $k^1$) represents the requirements on our matching once we have selected the even (resp., odd) edges of $D$. To apply recursion, we need to prove that, if the component $D$ is carefully chosen, at least one of $k^0$, $k^1$ belongs to the segment between $p^0$, $p^1$. 


 %

To fix ideas, suppose without loss of generality $r_1\ge r_0$ (implying $b_1\le b_0$), i.e, $p_{M_1}$ lies lower and to the right with respect to $p_{M_0}$. We further  distinguish two subcases.

\begin{enumerate}
\item
Assume first that, in all components, the odd edges are red and the even edges are blue, i.e., $r_1^D=b_0^D=h^D$ and $r_0^D=b_1^D=0$ for every component $D$, and also $r_1=b_0=\ell$ and $r_0=b_1=0$. Let us choose a component $D$ of smallest size. Then $h^D\le\ell/2$, and thus, as $k_R+k_B=\ell$, at least one of $k_R,k_B$ is ${}\ge h^D$.
Since  $k^0=(k_R,k_B-h^D)$ and $k^1=(k_R-h^D,k_B)$, we see that one of the points $k^0,k^1$ has non-negative components and hence lies on the segment between $p^0,p^1$. (See Figure \ref{fig:general}, left.) 

\item
We are left with the case where in some component $D$ the odd edges are blue and the even edges are red. Choosing that component, we see that actually both $k^0$ and $k^1$ lie on the segment between $p^0, p^1$: indeed, in this case, the segment between $p^0, p^1$ is exactly the segment containing $x-(h^D,0)$, $x-(0,h^D)$ for each $x$ in the segment between $p_{M_0,}p_{M_1}$. (See Figure \ref{fig:general}, right.)
\end{enumerate}
This completes the analysis for case 2.
\end{enumerate}

 Finally, to conclude the proof, we remark that the above arguments give an efficient algorithm to find a matching satisfying the requirements (where the algorithm will be recursive in case 2). Thus, a desired matching can be found efficiently.
\end{proof}

\begin{figure}[htbp]
    \centering
    \begin{tikzpicture}[scale=0.45]

    \draw[step=1, lightgray, very thin] (0,0) grid (10,10);

    \draw[thick,->] (0,0) -- (10.5,0) node[right] {$x_R$};
    \draw[thick,->] (0,0) -- (0,10.5) node[above] {$x_B$};

    \node[circle,fill=black, inner sep=1.5] (m2) at  (0,8) {};
     \node[right] at (m2) {$p_{M_0}$};

\node[circle,fill=black, inner sep=1.5] (m1) at  (8,0){};
    \node[below] at (m1) {$p_{M_1}$};

\draw (m1) -- (m2);

\node[circle,fill=red, inner sep=1.5] (k) at  (1,7){};
    \node[right] at (k) {$k$};

\node[circle,fill=black, inner sep=1.5] (p2) at  (0,6) {};
     \node[left] at (p2) {$p^{0}$};

\node[circle,fill=black, inner sep=1.5] (p1) at  (6,0){};
    \node[below] at (p1) {$p^{1}$};

\draw (p1) -- (p2);
    \node[circle,fill=red!40, inner sep=1.5] (k') at  (-1,7){};
    \node[left] at (k') {$k^1$};
    \node[circle,fill=red, inner sep=1.5] (k'') at  (1,5){};
    \node[left] at (k'') {$k^0$};


\end{tikzpicture}
 \begin{tikzpicture}[scale=0.45]

    \draw[step=1, lightgray, very thin] (0,0) grid (10,10);

    \draw[thick,->] (0,0) -- (10.5,0) node[right] {$x_R$};
    \draw[thick,->] (0,0) -- (0,10.5) node[above] {$x_B$};

    \node[circle,fill=black, inner sep=1.5] (m2) at  (2,8) {};
     \node[right] at (m2) {$p_{M_0}$};

\node[circle,fill=black, inner sep=1.5] (m1) at  (8,2){};
    \node[right] at (m1) {$p_{M_1}$};

\draw (m1) -- (m2);

\node[circle,fill=red, inner sep=1.5] (k) at  (3,7){};
    \node[right] at (k) {$k$};

\node[circle,fill=black, inner sep=1.5] (p2) at  (0,8) {};
     \node[left] at (p2) {$p^{0}$};

\node[circle,fill=black, inner sep=1.5] (p1) at  (8,0){};
    \node[below] at (p1) {$p^{1}$};

\draw (p1) -- (p2);
    \node[circle,fill=red, inner sep=1.5] (k') at  (1,7){};
    \node[left] at (k') {$k^0$};
    \node[circle,fill=red, inner sep=1.5] (k'') at  (3,5){};
    \node[left] at (k'') {$k^1$};


\end{tikzpicture}
    \caption{Examples for the proof of Lemma \ref{lem:general}. In the left figure, we have $r_1=b_0=8$, $r_0=b_1=0$, and $h^D=2$. In this case, the point $k^0$ is on the  segment between $p^0,p^1$. In the right figure, we have $r_1=b_0=8$, $r_0=b_1=2$, and $h^D=2$. Then $k^0$, $k^1$ are both on the segment between $p^0, p^1$. }
    \label{fig:general}
\end{figure}
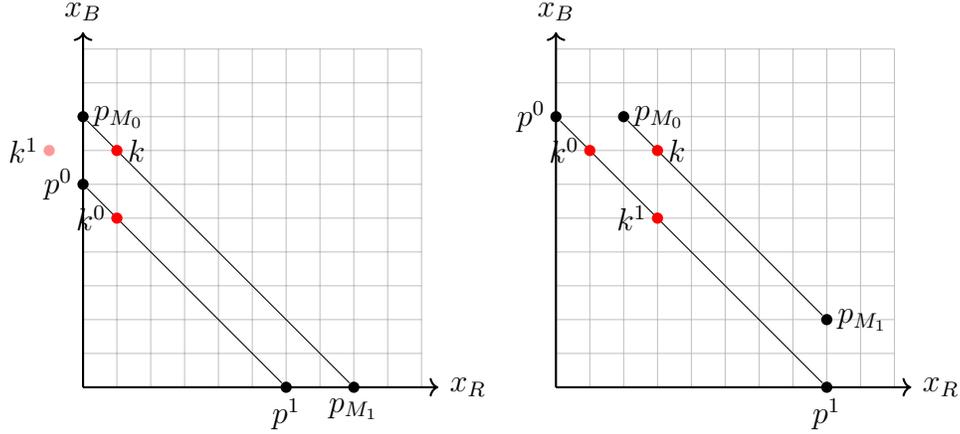

\begin{remark}\em
The above result is tight. Indeed, let $G$ be the union of two cycles $C$ and $C'$ of length 8, with color sequences R\,Y\,R\,Y\,R\,Y\,R\,Y and B\,Y\,B\,Y\,B\,Y\,B\,Y, respectively. Let $M_0$ be the perfect matching formed by all red and blue edges, and $M_1$ be the perfect matching formed by all yellow edges. Then $(k_R,k_B):=(2,2)$ is on the segment joining $p_{M_0}$ and $p_{M_1}$. It is immediate to see that no matching of $G$ whose restriction to $C$ or $C'$ is perfect can have $k_R$ red edges and $k_B$ or $k_B-1$ blue edges. Thus, any matching $M$ with $k_R$ red edges and $k_B$ or $k_B-1$ blue edges satisfies $|M|\le|M_1|-2$.
\end{remark}

\medskip
Before proving our main theorem, we need a final easy observation.

\begin{obs}\label{obs:parallelogram}
Let $P\subseteq\R^n$ be a polytope with 0/1 vertices. If $P$ has dimension 2, then $P$ is a triangle or a parallelogram.
\end{obs}

\begin{proof}
Suppose that $P$ has at least four vertices, one of which can be assumed to be the origin without loss of generality. That is, $P$ contains at least four pairwise distinct points $0,x,y,z$ with 0/1 coordinates. Since $P$ has dimension 2, there exist $\lambda,\mu\in\R$ such that $z=\lambda x+\mu y$. Let $i\in\{1,\dots,n\}$ be an index such that $x_i\ne y_i$, say $x_i=0$ and $y_i=1$ without loss of generality. Then $z_i=\mu$, and therefore $\mu\in\{0,1\}$. If $\mu=0$, we obtain $z=\lambda x$, which is possible only if $z\in\{0,x\}$, a contradiction. Thus $\mu=1$, hence $z=\lambda x+y$. As $x$ and $y$ differ in at least one component, this gives $\lambda\in\{-1,0,1\}$. But $\lambda\ne0$ because $z\ne y$, thus $\lambda\in\{-1,1\}$ and $z=x+y$ or $y=x+z$. In both cases, $P$ has exactly four vertices that form a parallelogram.
\end{proof}

We are now ready to show our result on general graphs, which we restate here for convenience. 

\begin{thm*}[\ref{thm:main}, restated]
There is a deterministic algorithm that,
given a red-blue-yellow colored graph $G$ (thus, with edge set $E=R\mathbin{\dot\cup} B\mathbin{\dot\cup} Y$) and integers $k_R, k_B\ge0$ such that the LP \[\max\{x(E): x\in P_M(G),\, x(R)=k_R,\, x(B)=k_B\}\] 
is feasible and has optimal value $\alpha^*$, finds a matching $M$ of $G$ with $|M|\geq \lfloor \alpha^*\rfloor-3$, exactly $k_R$ red edges, and $k_B$ or $k_B-1$ blue edges. 
\end{thm*}

\begin{proof}
Let $x^*$ be a basic optimal solution of the LP. Then $x^*$ belongs to a face $F$ of $P_M(G)$ of dimension at most two, and thus $F$ is a singleton, a segment, or a two-dimensional polytope.
As $P_M(G)$ (hence, $F$) has only 0/1 vertices, by Observation \ref{obs:parallelogram} in the latter case $F$ is a triangle or a parallelogram. By choosing a minimal face, we can assume that $x^*$ lies in the relative interior of $F$. 

Every vertex $x$ of $F$ is the characteristic vector of a matching $M$, and the objective value $x(E)$ at $x$ is equal to $|M|$. Furthermore, if $x,x'$ are adjacent vertices of $F$ (and thus of $P_M(G)$), the symmetric difference between the corresponding matchings is an alternating cycle or an alternating path, and thus $-1\le|M|-|M'|\le 1$.

Now, let $M^-,M^+$ be a smallest and a largest matching (respectively) corresponding to a vertex of $F$. By the above observations, if $F$ is a segment or a triangle, then $|M^+|\le|M^-|+1$, which implies $\floor{\alpha^*}=|M^-|$ because $x^*$ is in the relative interior of $F$; and if $F$ is a parallelogram, then $|M^+|\le|M^-|+2$, which implies $\floor{\alpha^*}\le|M^-|+1$. Thus, summarizing,
\begin{equation}\label{eq:bound-alpha}
\floor{\alpha^*}\le|M^-|+1,\mbox{ and }\floor{\alpha^*}=|M^-| \mbox{ if $F$ is not a parallelogram.}
\end{equation}

Let $\pi:\R^E\to\R^2$ be the linear function defined as follows: $\pi(x)=(x(R),x(B))$. Note that $\pi(x^*)=(k_R,k_B)$, and if $x$ is a vertex of $P_M(G)$, then $x$ is the characteristic vector of a matching $M$ such that $\pi(x)=p_M$.
Moreover, since $x^*$ is in the relative interior of $F$, $(k_R,k_B)$ is in the relative interior of $\pi(F)$.
Note also that although the dimension of $\pi(F)$ can be smaller than that of $F$, the linearity of $\pi$ implies that $\pi(F)$ is a singleton, a segment, a triangle, or a parallelogram. We consider each of these cases separately.

\begin{enumerate}
\item If $\pi(F)$ is a singleton, then $\pi(F)=\{(k_R,k_B)\}$, and thus every matching $M$ that corresponds to a vertex of $F$ satisfies $p_M=(k_R,k_B)$. If we choose a maximum matching $M$ among these, we have $|M|\ge\alpha^*$. 
    
\item If $\pi(F)$ is a segment, then there exist adjacent vertices $x^0,x^1$ of $F$ (corresponding to matchings $M_0,M_1$, respectively) such that $(k_R,k_B)$ is on the segment between $\pi(x^0)=p_{M_0}$ and $\pi(x^1)=p_{M_1}$. We assume $|M_0|\ge|M_1|$ without loss of generality. 
Furthermore, we can assume that $M_0\cap M_1=\emptyset$, as otherwise we can include in $M$ every edge of $M_0\cap M_1$, and translate $p_{M_0}$, $p_{M_1}$ and $(k_R,k_B)$ by an identical integer vector. Now, Theorem \ref{thm:cycle_paths_choice}  yields a matching $M$ that satisfies the color requirements and is of size $|M|\ge|M_1|-1\ge|M^-|-1\ge\floor{\alpha^*}-2$, where the last inequality follows from \eqref{eq:bound-alpha}. (With some care, one can actually choose adjacent vertices $x^0,x^1$ such that $\floor{\alpha^*}=|M_1|$ and thus show that $|M|\ge\floor{\alpha^*}-1$.) 

\item If $\pi(F)$ is a triangle, then $F$ is also a triangle with vertices corresponding to some distinct matchings $M_1,M_2,M_3$.  Thus, $\pi(F)$ has vertices $p_{M_1},p_{M_2},p_{M_3}$. 
Recall that $(k_R,k_B)$ lies in the relative interior of $\pi(F)$. (See Figure \ref{fig:main} for an illustration of this case.)

\begin{figure}[h!]
    \centering     \begin{tikzpicture}[scale=0.6]

    \draw[step=1, lightgray, very thin] (0,0) grid (10,10);

    \draw[->] (0,0) -- (10.5,0) node[right] {$x_R$};
    \draw[->] (0,0) -- (0,10.5) node[above] {$x_B$};

\draw[thin, gray] (5,6.5) -- (5,1.66);

\node[circle,fill=black, inner sep=1] (1) at  (4,1){};
\node[below] at (1) {$p_{M_1}$};

\node[circle,fill=black, inner sep=1] (2) at  (9,4){};
\node[right] at (2) {$p_{M_2}$};

\node[circle,fill=black, inner sep=1] (3) at  (1,9){};
\node[above] at (3) {$p_{M_3}$};

\node[circle,fill=red, inner sep=1.2] (k) at  (5,4){};
\node[right] at (k) {$(k_R, k_B)$};

\node[circle,fill=black, inner sep=0.9] (m'') at  (5,6){};
\node[left] at (m'') {\small $p_{M''}$};

\node[circle,fill=black, inner sep=0.9] (m') at  (5,2){};
\node[left] at (m') {\small$p_{M'}$};

\node[circle,fill=red, inner sep=0.7] (k'') at  (5,6.5){};
\node[right, shift={(0,0.1)}] at (k'') {\small$(k_R,k_B'')$};

\node[circle,fill=red, inner sep=0.7] (k') at  (5,1.66){};
\node[right, shift={(0,-0.25)}] at (k') {\small$(k_R,k_B')$};

\draw[gray] (1) -- (2) -- (3) -- (1);



\end{tikzpicture}
    \caption{  An example for case 3 of the proof of Theorem \ref{thm:main}.
    } \label{fig:main}
\end{figure}
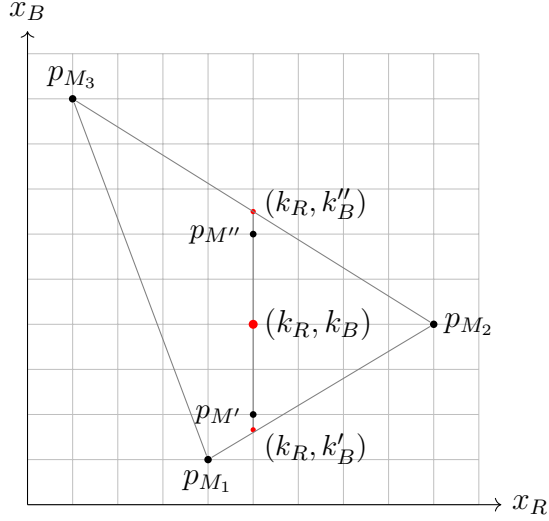

The line of equation $x_R=k_R$ intersects the boundary of $\pi(F)$ at precisely two points, which we denote by $(k_R, k'_B)$ and $(k_R, k''_B)$, with $k'_B<k''_B$. Without loss of generality, $(k_R,k'_B)$ belongs to the segment $p_{M_1}p_{M_2}$, and $(k_R,k''_B)$ belongs to the segment $p_{M_2}p_{M_3}$. We construct a matching $M'$ from $M_1\cup M_2$ as follows. First, by arguing as in the previous case, we can assume that $M_1\cap M_2=\emptyset$. 
Then, we apply Lemma \ref{lem:cycle_frac}: we obtain a matching $M'$ with $|M'|\geq \min\{|M_1|, |M_2|\}-1$, $k_R$ red edges, and at most $\lceil k'_B \rceil$ blue edges. 

Similarly, we obtain a matching $M''$, where $|M''|\geq \min\{|M_2|, |M_3|\}-1$ and $M''$ has $k_R$ red edges and at least $\lceil k''_B \rceil-1$ blue edges. 
As $k_B$ is an integer satisfying $k'_B<k_B<k''_B$, we have $\lceil k'_B \rceil\le k_B\le\lceil k''_B \rceil-1$, and thus the point $(k_R, k_B)$ lies on the segment $p_{M'}p_{M''}$. We now apply Lemma \ref{lem:general} to $M', M''$ and $(k_R, k_B)$ to obtain a matching that meets the color requirements and is of size at least $\min\{|M'|, |M''|\}-2\ge|M^-|-3=\lfloor \alpha^*\rfloor -3$, where the equation follows from \eqref{eq:bound-alpha}.

\item If $\pi(F)$ is a parallelogram, then $F$ is also a parallelogram with vertices $x^1,x^2,x^3,x^4$, where $x^i$ is adjacent to $x^{i+1}$ for each $i\in\{1,2,3,4\}$ (with index 5 identified with 1). Note that $x_1-x_2=x_4-x_3$, $x_2-x_3=x_1-x_4$, and these two vectors have disjoint support because they are consecutive edges of a parallelogram with 0/1 vertices. 
If we denote by $M_i$ the matching whose characteristic vector is $x^i$ (for $i\in\{1,2,3,4\}$), then $\pi(F)$ has vertices $p_{M_1},\dots,p_{M_4}$, in this order. Also, recall that $M_i\sd M_{i+1}$ is an alternating cycle or an alternating path for each $i$. Since $M_i\sd M_{i+1}$ coincides with the support of $x^i-x^{i+1}$ for each $i$, the above observations show that there is no common edge between $D':=M_1\sd M_2=M_3\sd M_4$ and $D'':=M_2\sd M_3=M_1\sd M_4$. 
We can actually show that $D'$ and $D''$ are also node-disjoint. In fact, assume that there exist two adjacent edges $e,f$ with $e\in D'$ and $f\in D''$, and suppose without loss of generality $e\in M_1\sm M_2$. We have the following chain of implications:
\[e\in M_1 \Rightarrow f\notin M_1 \Rightarrow f\in M_4 \Rightarrow e\notin M_4 \Rightarrow e\in M_3.\] 
This yields $e\in M_3\sm M_2\subseteq D''$, which contradicts the fact that $e\in D'$ and $D',D''$ are edge-disjoint. Thus, $D',D''$ are node-disjoint. Note also that $M_1,\dots,M_4$ coincide outside of $D'\cup D''$.

As already pointed out, the cardinalities of the largest and smallest matchings among $M_1,\dots,M_4$ differ by at most two, i.e., $|M^+|\le|M^-|+2$. We note here that if $|M^+|=|M^-|+2$, then $D'$ and $D''$ must be odd paths. Thus, $\floor{\alpha^*}\le|M^-|+1$ (as in \eqref{eq:bound-alpha}), with equality holding only if $D'$ and $D''$ are odd paths.

As in the previous case, we consider the line of equation $x_R=k_R$, which intersects the boundary of $\pi(F)$ at precisely two points $(k_R, k'_B)$ and $(k_R, k''_B)$, with $k'_B<k''_B$.
Without loss of generality, $(k_R, k'_B)$ is on the segment $p_{M_1}p_{M_2}$ and $(k_R, k''_B)$ is on the segment $p_{M_i}p_{M_j}$, where $(i,j)\in\{(2,3),(3,4)\}$.
By arguing as in the previous case, we obtain a matching $M'$ that has $k_R$ red edges and at most $\ceil{k'_B}$ blue edges, and satisfies $|M'|\ge\min\{|M_1|,|M_2|\}-1$. Similarly, we can construct a matching $M''$ that has $k_R$ red edges and at least $\ceil{k''_B}-1$ blue edges, and satisfies $|M''|\ge\min\{|M_i|,|M_j|\}-1$. Note that $M'$ coincides with $M_1$ and $M_2$ outside of $D'$. Similarly, if $(i,j)=(2,3)$ (resp., $(i,j)=(3,4)$), $M''$ coincides with $M_i$ and $M_j$ outside of $D''$ (resp., $D'$). Then, as $M_1,\dots,M_4$ coincide outside of $D'\cup D''$, we see that also $M',M''$ coincide outside of $D'\cup D''$, i.e.,
$M'\sd M''\subseteq D'\cup D''$.

Following again the arguments of the previous case, from Lemma \ref{lem:general} we obtain a matching $M$ that meets the color requirement and satisfies 
\[|M|\ge
\begin{cases}
\min\{|M'|,|M''|\}-2\ge|M^-|-3& \mbox{if $M'\sd M''$ contains a cycle,}\\
\min\{|M'|,|M''|\}-1\ge|M^-|-2 &\mbox{otherwise.}
\end{cases}\]
Now, if $\floor{\alpha^*}\le|M^-|$, then $|M|\ge\floor{\alpha^*}-3$. Otherwise, as shown above, $\floor{\alpha^*}=|M^-|+1$ and $D',D''$ must be node-disjoint odd paths. As $M'\sd M''\subseteq D'\cup D''$, $M'\sd M''$ contains no cycles and thus again $|M|\ge\floor{\alpha^*}-3$.
\end{enumerate}
This concludes the analysis of all cases and thus finishes the proof. 
\end{proof}

As mentioned at the end of Subsection \ref{sec:goodpaths},  a weaker version of Theorem \ref{thm:cycle_choice} can be obtained by exploiting the existence of an intersecting pair rather than a crossing pair. In this weaker version, the number of blue edges is allowed to be one unit smaller than $k_B$ also when $Y\ne\emptyset$. Furthermore, this is sufficient to prove Theorem \ref{thm:cycle_paths_choice}. By just going trough all the above proofs again, one can verify that in this relaxed setting Lemma \ref{lem:cycle_frac} can still be derived, while Lemma \ref{lem:general} becomes weaker: the number of blue edges can only be guaranteed to be between $k_B-2$ and $k_B$. With these results, one obtains a weaker version of Theorem \ref{thm:main}, where the matching can have at most two fewer blue edges than required.
\section{Crossing pairs and the Crossing Lemma}
\label{sec:crossing}

In Section \ref{sec:intro} we gave the notion of crossing pair under the assumption that $f^\infty$ divides the plane into two connected components. Although this is intuitively clear, we provide a formal proof here. We then prove  the Crossing Lemma (Lemma \ref{lem:crossing}).

The classical Jordan curve theorem asserts that every plane simple closed curve divides the plane into two connected components (a bounded one and an unbounded one). More formally, if $\gamma:\mathbb [0,1]\to\R^2$ is a continuous map such that $\gamma(0)=\gamma(1)$ and $\gamma|_{[0,1)}$ is injective (or, equivalently and more concisely, if $\gamma:\mathbb S^1\to\R^2$ is a continuous injective map, where $\bb S^1$ is the unit circle), then $\R^2\setminus\Im(\gamma)$ consists of two connected components. A slightly different version of this theorem states that every simple closed curve on the 2-dimensional sphere $\bb S^2$ divides $\bb S^2$ into two connected components, and is formalized as above with $\R^2$ replaced by $\bb S^2$. We refer the reader to \cite{munkres-topology} for the proofs of these results.

The following variant of the above theorems is certainly known, and just as intuitively clear. However, since we were unable to find a reference, we provide here a short proof.

\begin{lemma}\label{lemma:Jordan-open}
Every continuous injective map $\gamma:\R\to\R^2$ such that \\$\lim_{t\to\pm\infty}\|\gamma(t)\|=\infty$ divides $\R^2$ into two connected components.
\end{lemma}
\begin{proof}
Since $\R$ is homeomorphic to the interval $(0,1)$, and $\R^2$ is homeomorphic to $\bb S^2\sm\{N\}$ (where $N$ is the north pole) via  the stereographic projection,  $\gamma$ is equivalent to a curve $\gamma':(0,1)\to\mathbb S^2\setminus\{N\}$, with the property $\lim_{t\to 0}\gamma'(t)=\lim_{t\to 1}\gamma'(t)=N$. If we extend $\gamma'$ by setting $\gamma'(0)=\gamma'(1)=N$, we obtain a continuous map $\gamma':[0,1]\to\mathbb S^2$ such that $\gamma'|_{[0,1)}$ is injective. We conclude by applying the sphere version of Jordan curve theorem.
\end{proof}


Since, in the assumptions of the Crossing Lemma, $f^\infty$ is an injective map, it is clear that $\lim_{t\to\pm\infty}\|f^\infty(t)\|=\infty$. Thus, by Lemma \ref{lemma:Jordan-open}, $f^\infty$ divides the plane into two connected components, and the notion of crossing pair given in Section \ref{sec:intro} is defined properly.

\begin{remark}\label{rem:crossing}
\em The notion that $g$ crosses $f^\infty$ can be easily formalized even if $f$ and $g$ are only assumed to be continuous (with $f^\infty$ injective), without requiring them to be piecewise-linear: one only has to require that $g$ intersects both connected components of $\R^2\setminus\Im(f^\infty)$. Unfortunately, things become far more delicate when one wants to identify a point where the crossing happens. Indeed, on the one hand, it is clear from the definition that $(u,v)$ is a crossing pair for $(f^\infty,g)$ if and only if, informally speaking and ignoring possible overlappings, $v$ is the point where $g$ leaves one component of $\R^2\sm\Im(f^\infty)$ and enters the other component; thus, the term \lq\lq crossing'' appears to be appropriate. On the other hand, however, the situation is not always so clear: for instance, there are examples where $f^\infty(u)=g(v)$ but, for every $\epsilon>0$, each of $g((v-\epsilon,v))$ and $g((v,v+\epsilon))$ intersects both components of $\R^2\setminus\Im(f^\infty)$. In this situation, it is not even clear if one would say that $g$ crosses $f^\infty$ there. But for piecewise-linear functions, our definition precisely matches the intuitive notion of crossing.
\end{remark}

We are ready to prove the Crossing Lemma, which we restate here.

\begin{lemma*}[\ref{lem:crossing}, restated]
Let $f: [0,\tau]\rightarrow \R^2$ be a continuous piecewise-linear map such that $f^\infty$ is injective, and let $q$ be a point on the segment between $f(0)$ and $f(\tau)$ such that $q\not\in\Im(f)$.
Then there exists a crossing pair $(u,v)$ for $(f^\infty,f+q-f(0))$ such that $v<u<v+\tau$.
\end{lemma*}

\begin{proof}
We use the notation $f(t)=(f_1(t),f_2(t))$ to specify the components of $f(t)$. Moreover, we define the map $g:=f+q-f(0)$.

\begin{claim}\label{cl:assum}
The lemma holds in general if it holds whenever all the following additional conditions are satisfied: $f(0)=(0,0)$, $f_1(\tau)>0$, $f_2(\tau)=0$, and $\min_{t\in[0,\tau]}f_2(t)=\min_{t\in\R}f^\infty_2(t)=0$.
\end{claim}

\begin{cpf}
Since all the properties involved in the statement of the lemma are invariant under isometries, and since $f(0)\ne f(\tau)$ by injectivity, we can assume without loss of generality that $f_2(0)=f_2(\tau)$ and $f_1(0)<f_1(\tau)$ (by applying a rotation), $\min_{t\in[0,\tau]}f_2(t)=\min_{t\in\R}f^\infty_2(t)=0$ (vertical translation), and $f(a)=(0,0)$ for some $a\in[0,\tau)$ (horizontal translation). It remains to show that we can further assume $a=0$.

Define $\tilde f,\tilde g:\R\to\R^2$ by setting $\tilde f(t)=f^\infty(t+a)$ and $\tilde g(t)=g^\infty(t+a)$ for $t\in\R$. Note that $\tilde f(0)=(0,0)$, $\tilde f_1(\tau)=f_1(\tau)-f_1(0)>0$, $\tilde f_2(\tau)=0$, and $\min_{t\in\R}\tilde f_2(t)=0$. Thus, the restriction $\tilde f|_{[0,\tau]}$ and the point $\tilde q:=q-f(0)$ satisfy all the assumptions of both the lemma and the claim. Assuming that the lemma holds when all these conditions are met, and noting that $\big(\tilde f|_{[0,\tau]}\big)^\infty=\tilde f$, there is a crossing pair $(\tilde u,\tilde v)$ for $(\tilde f,\tilde g|_{[0,\tau]})$ such that $\tilde v<\tilde u<\tilde v+\tau$.
Now, if $\tilde v+a<\tau$, setting $u=\tilde u+a$ and $v= \tilde v+a$ gives $v<u<v+\tau$; since the two components of $\R^2$ determined by $\tilde f$ are the same as those determined by $f^\infty$, $(u,v)$ is the desired crossing pair for $(f^\infty,g)$. If $\tilde v+a>\tau$, by periodicity we can take $u=\tilde u+a-\tau$ and $v= \tilde v+a-\tau$. 

To conclude, we show that the case $\tilde v+a=\tau$ is impossible. Indeed, if $\tilde v+a=\tau$ then
$$g(\tau)=\tilde g(\tilde v)=\tilde f(\tilde u)=f(\tau)-f(0)+f(\tilde u+a-\tau),$$
where the first equality follows from the definition of $\tilde g$ and the fact that $0<\tilde v<\tau$, the second one holds because $(\tilde u,\tilde v)$ is a crossing pair for $(\tilde f,\tilde g|_{[0,\tau]})$, and the third one is by definition of $\tilde f$ and the fact that $\tau<\tilde u+a<2\tau$ (as $\tilde v<\tilde u<\tilde v+\tau$). It follows that $q=g(\tau)-f(\tau)+f(0)=f(\tilde u+a-\tau)$, contradicting the assumption $q\not\in \Im(f)$.
\end{cpf}

In the remainder of the proof, we assume that the properties of Claim \ref{cl:assum} are satisfied. In particular, this implies that $q_2=0$ and $f^\infty_2
(k\tau)=0$ for each $k\in \Z$.

\begin{claim}\label{cl:ordered zeros}
If $t_1,t_2\in\R$ satisfy $t_1<t_2$ and $f^\infty_2(t_1)=f^\infty_2(t_2)=0$, then $f^\infty_1(t_1)<f^\infty_1(t_2)$.
\end{claim}

\begin{cpf}
Since $\min_{t\in\R}f^\infty_2(t)=0$, if $f^\infty_2(\bar t)=0$ for some $\bar t$ that is not a breakpoint of $f^\infty$, then $f^\infty_2(t)=0$ for every $t$ between the two breakpoints immediately preceding and following $\bar t$. Therefore, it suffices to prove the claim under the assumption that $t_1$ and $t_2$ are both breakpoints.

If the claim does not hold, then there exist breakpoints $t_1$ and $t_2$ such that $t_1<t_2$, $f^\infty_2(t_1)=f^\infty_2(t_2)=0$, and $f^\infty_1(t_1)>f^\infty_1(t_2)$ (where the last inequality is strict by the injectivity of $f^\infty$). Since $f^\infty_1(0)<f^\infty_1(\tau)$, $t_0:=t_2-\tau$ satisfies $f^\infty_2(t_0)=0$ and $f^\infty_1(t_0)<f^\infty_1(t_2)$.
Now, define a closed curve $\gamma$ by concatenating the following two pieces: (i) $f^\infty([t_0,t_1])$; (ii) any arc joining $f^\infty(t_1)$ and $f^\infty(t_0)$ fully contained in the halfplane $x_2<0$, except for its extreme points. By Jordan's theorem, $\gamma$ divides the plane into two connected components ---a bounded one and an unbounded one. The point $f^\infty(t_2)$ is in the bounded region. If we take any $t_3>t_2$ such that $f^\infty(t_3)$ is in the unbounded region, the curve $f^\infty([t_2,t_3])$ must cross $\gamma$. As $f^\infty_2(t)\ge0$ for every $t$, $f^\infty([t_2,t_3])$ crosses the piece (i) of $\gamma$ (i.e., $f^\infty([t_0,t_1])$), contradicting the injectivity of $f^\infty$.
\end{cpf}

Given any two breakpoints $t_1,t_2$ of $f^\infty$ such that $t_1<t_2$, $f^\infty_2(t_1)=f^\infty_2(t_2)=0$ and $f^\infty_2(t)>0$ for every {\em breakpoint} $t\in(t_1,t_2)$, we construct a closed curve by concatenating the following two pieces: (i) $f^\infty([t_1,t_2])$; (ii) any arc joining $f^\infty(t_2)$ and $f^\infty(t_1)$ fully contained in the halfplane $x_2<0$ and in the vertical strip $f^\infty_1(t_1)<x_1<f^\infty_1(t_2)$, except for its extreme points. We call {\em bubble} the topologically closed bounded region determined by this closed curve. We say that (i) is the {\em upper boundary} of the bubble. Note that either $f^\infty_2(t)>0$ for every $t\in(t_1,t_2)$ or $f^\infty_2(t)=0$ for every $t\in[t_1,t_2]$. In the former case we say that the bubble is {\em fat}, in the latter case we call it {\em slim}. The points $f^\infty(t_1)$ and $f^\infty(t_2)$ are respectively the left and right extreme of the bubble. By the previous claim and the injectivity of $f^\infty$, two distinct bubbles share at most one point, and if they share a point, this is an extreme of both of them. The bubbles are naturally ordered from left to right, based on the position of their extremes.

Before proceeding, we need the following fact:

\begin{claim}\label{cl:circle}
Let $I\subseteq\R$ and $J\subseteq[0,\tau]$ be two closed intervals that coincide modulo $\tau$ (i.e., $\{z \mod \tau : z\in I \}=\{z \mod \tau : z\in J \})$,
where $J$ contains 0 or $\tau$, and define
$Z_{f,I}=f^\infty(I)\cap \{x\in\R^2:x_2=0\}$, $Z_{g,J} = g(J)\cap \{x\in\R^2:x_2=0\}$. Then $Z_{f,I}\not\subseteq Z_{g,J}$.
\end{claim}

\begin{cpf}
To simplify the notation, in this proof we write $Z_f$ and $Z_g$ for $Z_{f,I}$ and $Z_{g,J}$, respectively.
Since $Z_f$ and $Z_g$ are contained in the $x_1$ axis, we will ignore the second component and view these sets as subsets of $\R$. It will also be useful to embed these sets in $\R/a\Z$ (the circle of length $a$), where $a=f_1(\tau)-f_1(0)=f_1(\tau)$. Thus, given a number $w\in\R$, we write $[w]$ to denote $w \mod a$; similarly, for a subset $S\subseteq\R$, we write $[S]=\{[w]:w\in S\}$.

Assume by contradiction $Z_f\subseteq Z_g$, which implies $[Z_f]\subseteq [Z_g]$. As $I=J$ modulo $\tau$ and $g(t)=f(t)+q$ for every $t\in[0,\tau]$, we have $[Z_g]=[Z_f+q_1]$ and therefore $[Z_f]\subseteq [Z_f+q_1]$. Since $[Z_f]$ and $[Z_f+q_1]$ have the same 1-dimensional measure in $\R/a\Z$, this implies that $[Z_f+q_1]\setminus [Z_f]$ has zero measure in $\R/a\Z$. Then, as $Z_f$ is a finite union of (possibly degenerate) closed intervals, $[Z_f+q_1]\setminus [Z_f]$ can only contain some isolated points. But the number of isolated points in $[Z_f]$ and $[Z_f+q_1]$ is the same, so the two sets coincide. Thus $[Z_f]=[Z_g]$. Since $J$ contains 0 or $\tau$ and $[q_1]=[g_1(0)]=[g_1(\tau)]$, this implies that $[q_1]\in [Z_f]$. In other words, also using $q_2=0$, there are some $t\in I$ and $k\in\Z$ such that $f^\infty(t)=q+k(a,0)=q+kf(\tau)$. Since $f^\infty_2
(k\tau)=f^\infty_2
((k+1)\tau)=0$, by Claim \ref{cl:ordered zeros} we find that $t\in [k\tau, (k+1)\tau]$, thus, by definition of $f^\infty$, $q\in f([0,\tau])$. This contradicts the assumptions of the lemma.
\end{cpf}

\begin{claim}\label{cl:first-pair}
There exists a crossing pair $(u,v)$ for $(f^\infty,g)$ with $0<u<2\tau$.
\end{claim}

\begin{cpf}
Since $g(0)=q\notin f([0,\tau])$, $g(0)$ is in the interior of a fat bubble; due to the periodicity of $f^\infty$, $g(\tau)$ is also in the interior of a fat bubble. 
Let $B_0$ be the bubble containing $g(0)$, and order the subsequent bubbles from left to right as $B_1,\dots, B_k$, where $B_k$ contains $g(\tau)$.  Clearly, $k\ge1$.
Since $g_1(0)>f_1^\infty(0)$, $g_1(\tau)<f_1^\infty(2\tau)$, and $f^\infty_2(0)=f^\infty_2(2\tau)=0$, the upper boundaries of $B_0,\dots,B_k$ are contained in $f^\infty([0,2\tau])$.

Assume by contradiction that there is no crossing pair $(u,v)$ for $(f^\infty,g)$ with $0<u<2\tau$. Then $g$ can cross the boundary of the bubbles $B_0,\dots,B_k$ only at their extremes. Furthermore, by Claim \ref{cl:ordered zeros} (which also applies to $g$  because $g$ is obtained by translating $f$), $\Im(g)$ cannot contain the left extreme of $B_0$. This implies that $g$ goes through all the bubbles $B_0,\dots, B_k$ in this order.  
In particular, $\Im(g)$ contains the right extreme of each of $B_0,\dots,B_{k-1}$, and if $B_i$ s a slim bubble for some $i\in\{1,\dots,k-1\}$ then $\Im(g)$ contains the segment joining its extremes.

To verify that this is not possible, we now rephrase the above property. Let $p$ be any point on the upper boundary of $B_0$, where $p$ is not an extreme of $B_0$,
and let $s$ be such that $f^\infty(s)=p$. Note that $f^\infty(s+\tau)$ is on the upper boundary of $B_k$. 
Then, defining $I=[s,s+\tau]$ and $J=[0,\tau]$, and using the notation of Claim \ref{cl:circle}, the property stated in the previous paragraph can be expressed as $Z_{f,I}\subseteq Z_{g,J}$, which contradicts Claim \ref{cl:circle}.
\end{cpf}

\begin{claim}
If there is no crossing pair $(u,v)$ for $(f^\infty,g)$ with $v<u\le\tau$, then there is a crossing pair with $\tau<u<2\tau$.
\end{claim}

\begin{cpf}
Assume by contradiction that no crossing pair satisfies $v<u\le\tau$ or $\tau<u<2\tau$. Then every crossing pair with $0<u<2\tau$ also satisfies $u\le v$. By Claim \ref{cl:first-pair}, there exists a crossing pair with this property, and we choose the one maximizing $u$. 

Consider the list of consecutive bubbles $B_0,\dots,B_k$, where the upper boundary of $B_0$ contains $f^\infty(u)=g(v)$, and $B_k$ is the (fat) bubble containing $g(\tau)$ in its interior. In case $u$ is a breakpoint with $f^\infty_2(u)=0$ (i.e., $f^\infty(u)$ is an extreme of two bubbles), we number the bubbles so that $f^\infty(u)$ is an extreme of both $B_0$ and $B_1$. Note that $k\ge1$, because the right extreme of $B_0$ is at most $\tau$ (as $u\le v<\tau$) and the left extreme of $B_k$ is at least $\tau$ (as $g_1(\tau)>f^\infty_1(\tau)$). 

By Claim \ref{cl:ordered zeros} and the fact that $g(\tau)$ belongs to the segment joining $f^\infty(\tau)$ and $f^\infty(2\tau)$, the breakpoints of $f^\infty$ delimiting $B_k$ belong to the interval $[\tau,2\tau]$. Therefore the upper boundaries of $B_1,\dots,B_k$ are contained in $f^\infty([u,2\tau])$.
Furthermore, by the maximality of $u$, $g$ does not cross $f^\infty|_{(u,2\tau)}$. Then $g$ can only enter and leave the bubbles $B_1,\dots,B_k$ at their extremes. In particular, as $g(\tau)$ is in $B_k$, $g$ enters $B_k$ at one of its extremes. However, Claim \ref{cl:ordered zeros} 
applied to $g$ implies that $g$ enters $B_k$ at its left extreme. By proceeding backwards, we see that $g$ goes through all the bubbles $B_1,\dots,B_k$ in this order. In particular, for every $t\in[u,\tau]$ such that $f^\infty_2(t)=0$ there exists $t'\in[v,\tau]\subseteq[u,\tau]$ such that  $g(t')=f^\infty(t)$. 
This contradicts Claim \ref{cl:circle} with $I=J=[u,\tau]$. 
\end{cpf}

A symmetric argument shows the following:

\begin{claim}
If there is no crossing pair $(u,v)$ for $(f^\infty,g)$ with $\tau\le u<v+\tau$, then there is a crossing pair with $0<u<\tau$.
\end{claim}

\begin{cpf}
Assume that no crossing pair satisfies $\tau\le u<v+\tau$ or $0<u<\tau$. Then every crossing pair with $0<u<2\tau$ also satisfies $u\ge v+\tau$. By Claim \ref{cl:first-pair}, there exists a crossing pair with this property, and we choose the one minimizing $u$.

Consider the list of consecutive bubbles $B_0,\dots,B_{k}$, where $B_0$ is the (fat) bubble containing $g(0)$ in its interior, and the upper boundary of $B_{k}$ contains $f^\infty(u)=g(v)$. In case $u$ is a breakpoint with $f^\infty_2(u)=0$, we number the bubbles so that $f^\infty(u)$ is an extreme of both $B_{k-1}$ and $B_k$. Note that $k\ge1$, because the right extreme of $B_0$ is at most $\tau$ (as $g_1(0)<f^\infty_1(\tau)$) and the left extreme of $B_k$ is at least $\tau$ (as $u\ge v+\tau>\tau$). 

Since the breakpoints of $f^\infty$ delimiting $B_0$ belong to the interval $[0,\tau]$, the upper boundaries of $B_0,\dots,B_{k-1}$ are contained in $f^\infty([0,u])$.
Furthermore, by the minimality of $u$, $g$ does not cross $f^\infty|_{(0,u)}$. Then $g$ can only enter and leave the bubbles $B_0,\dots,B_{k-1}$ at their extremes. Therefore, as $g(0)$ is in $B_0$, $g$ leaves $B_0$ at its right extreme and then goes through all the bubbles $B_1,\dots,B_{k-1}$ in this order. In particular, for every $t\in[\tau,u]$ such that $f^\infty_2(t)=0$ there exists $t'\in[0,v]\subseteq[0,u-\tau]$ such that $g(t')=f^\infty(t)$. This contradicts Claim \ref{cl:circle} with $I=[\tau,u]$ and $J=[0,u-\tau]$. 
\end{cpf}

Now, if there is a crossing pair with $v<u\le\tau$ or $\tau\le u<v+\tau$, then we are done, as the inequalities $v<u<v+\tau$ are satisfied. Thus, we assume that this is not the case. Then, by the two claims above, there are two crossing pairs $(u,v)$ and $(u',v')$ with $0<u<\tau<u'<2\tau$, and we choose them with $u$ maximal and $u'$ minimal. It follows that $g$ does not cross $f^\infty|_{[u,u']}$.

If $v<u$ or $u'<v'+\tau$, we are done again. Therefore in the following we assume $v\ge u$ and $u'\ge v'+\tau$. Note further that since $u<u'$, the two crossing pairs do not coincide, and thus we also have $v\ne v'$.

Assume first that $v'<v$, and let $(\bar u,\bar v)$ be the intersecting pair for $(f^\infty,g)$ with $\tau< \bar u\le u'$, $0<\bar v\le v'$, and $\bar v$ minimal with respect to these properties. (This is well defined, as in the worst case one has $(\bar u,\bar v)=(u',v')$.) Let $\gamma$ be the curve obtained by concatenating $g([0,\bar v])$, $f^\infty([\tau,\bar u])$, and an arc joining $f^\infty(\tau)$ and $g(0)$ contained in the halfplane $x_2<0$, except for its extremes (see Figure \ref{fig:crossing_v'<v}). This is a closed curve because $f^\infty(\bar u)=g(\bar v)$, and it is simple by the minimality of $\bar v$. 
For $\epsilon>0$ sufficiently small, $f^\infty(\tau-\epsilon)$ is in the bounded region delimited by $\gamma$, otherwise $f^\infty([0,\tau])$ would intersect $f^\infty([\bar u, 2\tau])$, contradicting the injectivity of $f^\infty$. Because $g$ does not cross $f^\infty|_{[u,u']}$, this implies that $f^\infty|_{[u,\tau]}$ is entirely contained in this bounded region, and therefore so is, in particular, the point $f^\infty(u)=g(v)$.
On the other hand, $g(\tau)$ is in the unbounded region, because $g_1(\tau)>f^\infty_1(\tau)$. 
Thus $g|_{[v,\tau]}$ must cross $\gamma$, which is not possible since $g$ is injective (because so is $f$ and $g$ is a translation of $f$) and does not cross $f^\infty|_{[u,u']}$.

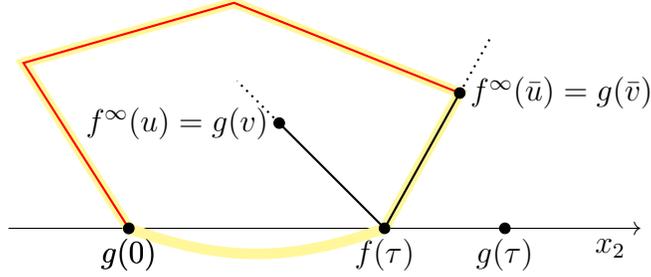
\begin{figure}
    \centering
\begin{tikzpicture}[scale=2]

    \draw [->] (-0.8,0) -- (3.4,0);
    \node[below] at (3.2,0) {$x_2$};

    \filldraw (0,0) circle (1pt) node[below] {$g(0)$};

     \coordinate (ftau) at (1.7,0);

    \coordinate (A) at (-0.7,1.1);
    \coordinate (B) at (0.7,1.5);
    \coordinate (C) at (2.2,0.9);
    \coordinate (I) at (1,0.7);

    \draw[line width=4pt, yellow!50, line cap=round] (0,0) -- (A) -- (B) -- (C) -- (ftau) to[bend left=20] (0,0);

    \filldraw (ftau) circle (1pt) node[below] {$f(\tau)$};
   
    \filldraw (2.5,0) circle (1pt) node[below] {$g(\tau)$};

    \draw[red, thick] (0,0) -- (A) -- (B) -- (C);

    \draw[thick] (C) -- (ftau) -- (I);
    \draw[thick, dotted] (C) -- ($  1.4*(C) - 0.4*(ftau) $);
    \draw[thick, dotted] (I) -- ($  1.4*(I) - 0.4*(ftau) $);

    \filldraw (C) circle (1pt);
    \filldraw (I) circle (1pt);

    \filldraw (0,0) circle (1pt) node[below] {$g(0)$};
    \node at (I) [left] {$f^\infty(u) = g(v)$};
    \node at (C) [right] {$f^\infty(\bar u) = g(\bar v)$};

\end{tikzpicture}
\caption{Case $v'<v$ of the proof of Lemma \ref{lem:crossing}. The relevant pieces of $f^\infty$ (resp., $g$) are shown in black (resp., red), while $\gamma$ is shown in yellow. }\label{fig:crossing_v'<v}
\end{figure}

Assume now that $v<v'$. 
Let $\gamma$ be the simple closed curve obtained by gluing together $g([0,\tau])$ and an arc joining $g(0)$ and $g(\tau)$ contained in the halfplane $x_2<0$, except for its extremes. We assign to $\gamma$ the orientation induced by $g$. Thus, as $g_1(0)<g_1(\tau)$, $\gamma$ is oriented clockwise, which means that the bounded (resp., unbounded) region lies on the right (resp., left) side of $\gamma$.
Since $g$ does not cross $f^\infty|_{[u,u']}$, we have that $f^\infty([u,u'])$ is fully contained in either the bounded or the unbounded region, where we consider these as closed regions (i.e., we include $\gamma$ in both of them). We analyze these two cases separately.

\begin{itemize}
\item Suppose first that $f^\infty([u,u'])$ lies in the unbounded region delimited by $\gamma$.
Given any $t\in[u,u']$ such that $g_1(0)<f^\infty_1(t)<g_1(\tau)$ and $f_2^\infty(t)=0$, the curve $g$ must contain the point $f^\infty(t)$, otherwise $f^\infty(t)$  would be in the interior of the bounded region. In particular, this is true for $t=\tau$, and we let $s\in[0,\tau]$ be such that $g(s)=f^\infty(\tau)$. Note that $s\ne v$, as $g(v)=f^\infty(u)$ and $u\ne\tau$.

If $s<v$ (i.e., $g$ passes through $f^\infty(\tau)$ before reaching $g(v)=f^\infty(u)$), then $g$ must go through all the bubbles formed by $f^\infty|_{[\tau,u']}$ before reaching $g(v)$. This implies that $g([s,v])$ contains all the points of $f^\infty([\tau,u'])$ that lie on the $x_1$ axis. 
As $s\ge0$ and $v< v'\le u'-\tau$, the latter still holds if $g([s,v])$ is replaced by $g([0,u'-\tau])$.
This contradicts Claim \ref{cl:circle} with $I=[\tau,u']$ and $J=[0,u'-\tau]$.

Otherwise, when $s>v$ (i.e., $g$ passes through $g(v)=f^\infty(u)$ before reaching $f^\infty(\tau)$), after visiting $g(v)$, $g$ must go through all the bubbles formed by $f^\infty|_{[u,\tau]}$ before reaching $f^\infty(\tau)$, and then must go through all the bubbles formed by $f^\infty|_{[\tau,u']}$ before reaching $g(v')=f^\infty(u')$. Thus,  $g([v,v'])$ (and therefore also $g([0,\tau])$) contains all the points of $f^\infty([u,u'])$ that lie on the $x_1$ axis. As $u+\tau\le v+\tau<v'+\tau\le u'$, the latter still holds if $f^\infty([u,u'])$ is replaced by $f^\infty([u,u+\tau])$. This contradicts Claim \ref{cl:circle} with $I=[0,\tau]$ and $J=[u,u+\tau]$.

\item Suppose now that $f^\infty([u,u'])$ lies in the bounded region delimited by $\gamma$. As above, $u+\tau< u'$, and thus $f^\infty(u+\tau)$ is in the bounded region. Note that $f^\infty(u+\tau)=g^\infty(v+\tau)$ by the periodicity of $f^\infty$ and $g^\infty$. Since $g^\infty$ is injective (as it is a translation of $f^\infty$, which is injective),  $f^\infty(u+\tau)=g^\infty(v+\tau)$ lies {\em in the interior} of the bounded region, and thus, as $\lim_{t\to\pm\infty}\|g^\infty(t)\|=\infty$, $g^\infty$ must enter and leave the bounded region, i.e., $g^\infty$ must cross $\gamma$. This implies that $g^\infty$ is not injective, a contradiction. 
\end{itemize}

This concludes the proof of the Crossing Lemma.
\end{proof}

\section{Open questions and possible extensions}
\label{sec:final}

The main result of this paper (Theorem \ref{thm:main}) gives a deterministic algorithm that, for any instance of the red-blue-yellow matching problem, returns a matching that satisfies exactly one color requirement (say red) and almost exactly another color requirement (i.e., missing at most one blue edge), and has large size (i.e., a large number of yellow edges), provided that the associated LP is feasible. This leaves a number of questions open. 

First, the associated LP tells us something about the matchings that satisfy exactly both color constraints: if it has optimal value $\alpha^*$, then this is an upper bound on the size of any such matching, and if instead the LP
is not feasible, then no such matching exists. However, solving this LP does not say anything about the existence of matchings satisfying the red constraint exactly and missing one blue edge. This drawback lies at the heart of the difficulty of solving similar problems (including the original red-blue matching problem) via linear programming approaches, which focus on matchings that are ``close'' to the optimal fractional solution and ignore further, potentially better solutions. We stress that problems such as deciding whether two matchings are at distance two in the matching polytope are NP-complete even for bipartite graphs, see, for instance, \cite{cardinal2025inapproximability}. We believe that the difficulty in ``navigating'' the matching polytope is a fundamental barrier that deterministic approaches would have to overcome in order to solve the red-blue matching problem.

Second, one can wonder whether the lower bound on the size of the matching obtained by our algorithm can be improved. Given that the bounds in some of the intermediate results (Theorem \ref{thm:cycle_paths_choice}, Lemma \ref{lem:general}) are best possible, improving Theorem \ref{thm:main} would require either applying these results in a radically different way, or, we suspect, completely new techniques. In particular, instead of studying pairs of matchings in the face containing the optimal fractional solution as in the proof of Theorem \ref{thm:main}, it would be interesting to directly consider the union of all such matchings: in the most interesting case, the instance obtained is a 3-regular graph that is uniquely 3-edge colorable. What can we say about large, colored matchings in those graphs? This leads us to the simplest non-trivial version of the red-blue matching problem: does it admit a deterministic algorithm if we restrict the input graph to be 3-regular, or even uniquely 3-edge-colorable? 

 Finally, it is natural to try to extend our results to the setting of $h$-colored graphs, for a constant $h\geq 4$. By applying the results from \cite{grandoni2014new} in a similar way as described at the end of Section \ref{sec:intro}, one can get a matching that is not too far away from the optimal solution in terms of size and color requirements; however, the error grows roughly as $h^2$. By a refining similar to that done in this paper, one might be able to obtain an analogue of Theorem \ref{thm:main} where, after solving a single LP, one gets a matching that is reasonably close to the optimum and almost satisfies all color requirements. To this end,  
  we formulate the following conjecture, that can be seen as a generalization of Theorem \ref{thm:cycle_choice}.  To a set of edges $M$ that is colored with (at most) $h$ colors, we associate a vector $p_M\in \Z_+^h$ whose $i$-th coordinate is the number of edges of color $i$ in $M$. Notice that this slightly differs from the notation used in the rest of the paper: before we had $h=3$ but only considered two coordinates (red and blue). However, it is easy to see that, for $h=3$, the following statement is implied by Theorem \ref{thm:cycle_choice}:

  \begin{conj}\label{conj}
     Let $G$ be a cycle of length $2\ell$ colored with $h$ colors, and denote by $M_0,M_1$ the perfect matchings of even and odd edges, respectively. Let $k\in \Z^h$ be a point on the segment between $p_{M_0}$ and $p_{M_1}$. Then, there is matching $M$ of size at least $\ell-h+2$ such that $p_M\leq k$ (component-wise) and $\|p_M-k\|_1\le h-2$.
 \end{conj} 

If the conjecture is true, it could be conceivable to prove a close analogue of Theorem \ref{thm:main} for a general constant number of colors.

 We recall that for the budgeted version of the problem, where color requirements are replaced by (a constant number of) general linear inequalities, polynomial-time approximation schemes are known (\cite{chekuri2011multi, grandoni2014new}). In particular, \cite{grandoni2014new} uses a geometrical existence result from Stromquist and Woodal \cite{stromquist1985sets}, coming from the Ham Sandwich Theorem, that seems to serve a similar role as our Crossing Lemma (Lemma \ref{lem:crossing}), and could help in our setting as well. 
 As a first step in this direction, we observe that the following holds:

\begin{lemma}\label{lem:sequences}\footnote{For $h=2$, this statement appears to be similar in spirit to the Cycle Lemma \cite{dershowitz1990cycle}.}
    Let $S$ be a cyclic sequence of colors chosen from a set of $d\geq 2$ colors, and let $k\in \Z^{d}$ be a point in the segment between the origin and $p_S$. Then $S$ contains a set $S'$ that is the union of $d-1$ intervals and satisfies $p_{S'} = k$.
\end{lemma}

To prove the above result, we use the following theorem.  (We call {\em circle} a nondegenerate closed interval with its endpoint identified. Below, it is implicit that all measures on a circle $C$ are defined at least on the Borel subsets of $C$, with respect to its standard topology. Also, recall that a probability measure $\mu$ is \emph{non-atomic} if $\mu(\{x\})=0$ for any $x$.)

\begin{thm}[Stromquist and Woodal \cite{stromquist1985sets}]\label{thm:measures}
    Let $d\geq 2$ and let $\mu_1,\dots,\mu_d$, be non-atomic probability measures on a circle $C$. For each $\lambda \in [0, 1]$, there exists a subset $K\subseteq C$ that is the union of at most $d - 1$ intervals and satisfies $\mu_i(K) = \lambda$ for each $i\in\{1,\dots,d\}$. 
\end{thm}

\begin{proof}[Proof of Lemma \ref{lem:sequences}]
Denote by $n$ the length of $S$ and by $n_i$ the number of elements of $S$ that have color $i$, for each $i\in\{1,\dots,d\}$. In other words, $n_i$ is the $i$-th component of $p_S$. Without loss of generality, we assume $n_i\ge1$ for each $i$. Let $C$ be the circle of length $n$.

We partition $C$ into $n$ sub-intervals $I_1,\dots, I_n$ of unit length, where each $I_j$ is colored as the $j$-th element of $S$. For each color $i\in\{1,\dots, d\}$, let $\mu_i$ be the measure on $C$ such that $\mu_i(X)$ is the length of the portion of $X$ that has color $i$ (i.e., the length of the intersection of $X$ with the union of intervals of color $i$). Each $\mu_i$ is clearly a non-atomic measure, and $\tilde \mu_i:=\mu_i/n_i$ is a non-atomic probability measure. If $\lambda\in [0,1]$ is such that $k=\lambda p_S$,  Theorem \ref{thm:measures} gives a subset $K\subseteq C$ that is the union of at most $d - 1$ intervals and satisfies $\tilde\mu_i(K) = \lambda$ for each $i$. Thus, $\mu_i(K)=n_i\lambda=k_i$. Now, we could just take $S'=\{j\in S: I_j\subseteq K\}$ and we would be done if each $I_j$ were either disjoint from $K$ or contained in $K$. However, this property is easily attainable. Indeed, whenever one of the intervals that form $K$ starts or ends with a proper portion of a sub-interval $I_j$, say of color $i$, the integrality of $k_i$ implies that there are other sub-intervals of color $i$ that are only partially contained in $K$. Removing all these sub-intervals and including all of $I_j$ in $K$ does not increase the number of intervals that form $K$, while preserving the color count. Iterating this, we obtain the aforementioned property and thus our desired set $S'$.
\end{proof}

Lemma \ref{lem:sequences} implies the validity of Conjecture \ref{conj} in the special case where all edges of $M_0$ have colors $1,\dots, h-1$ and all edges of $M_1$ have color $h$. Indeed, applying Lemma \ref{lem:sequences} to the sequence of colors in $M_0$ (with $d=h-1$ and $k$ restricted to the first $d$ coordinates) gives a subset of $M_0$ with the desired colors.
This corresponds to the union of at most $h-2$ paths in $G$, and by considering the symmetric difference between $M_1$ and all these paths we obtain a matching $M$ satisfying Conjecture \ref{conj}.
 
\section*{Acknowledgements.}
The authors are deeply indebted to Michele Conforti for helpful discussions that ignited and nurtured the project. We are also grateful to two anonymous reviewers for their careful reading of this rather technical work. Their useful suggestions improved the presentation of the paper.

This work has been partially funded by the SID project ``Network optimization problems under extra constraints'' under the BIRD 2024 program sponsored by the University of Padua.

The work of M. Di Summa has been partially funded by the European Union - NextGenerationEU
under the National Recovery and Resilience Plan (NRRP), Mission 4 Component 2 Investment
1.1 - Call PRIN 2022 No. 104 of February 2, 2022 of Italian Ministry of University and Research;
Project 2022BMBW2A (subject area: PE - Physical Sciences and Engineering) “Large-scale optimization
for sustainable and resilient energy system'', CUP I53D23002310006. 
\printbibliography

@article{papadimitriou1982complexity,
  title={The complexity of restricted spanning tree problems},
  author={Papadimitriou, Christos H and Yannakakis, Mihalis},
  journal={Journal of the ACM (JACM)},
  volume={29},
  number={2},
  pages={285--309},
  year={1982},
  publisher={ACM New York, NY, USA},
  doi={10.1145/322307.322309}
}

@inproceedings{doron2023eptas,
  title={An EPTAS for Budgeted Matching and Budgeted Matroid Intersection via Representative Sets},
  author={Doron-Arad, Ilan and Kulik, Ariel and Shachnai, Hadas},
  booktitle={International Colloquium on Automata, Languages and Programming},
  pages={49:1-49:16},
  year={2023},
  doi={10.4230/LIPIcs.ICALP.2023.49},
}

@inproceedings{el2023exact,
  title={Exact Matching: Algorithms and Related Problems},
  author={El Maalouly, Nicolas},
  booktitle={40th International Symposium on Theoretical Aspects of Computer Science (STACS 2023)},
  pages={29:1--29:17},
  year={2023},
  doi={10.4230/LIPIcs.STACS.2023.29},
}

@inproceedings{durr2023approximation,
  title={An Approximation Algorithm for the Exact Matching Problem in Bipartite Graphs},
  author={D{\"u}rr, Anita and El Maalouly, Nicolas and Wulf, Lasse},
  booktitle={Approximation, Randomization, and Combinatorial Optimization. Algorithms and Techniques (APPROX/RANDOM 2023)},
  pages={18:1--18:21},
  year={2023},
  doi={10.4230/LIPIcs.APPROX/RANDOM.2023.18}
}

@article{mulmuley1987matching,
  title={Matching is as easy as matrix inversion},
  author={Mulmuley, Ketan and Vazirani, Umesh V and Vazirani, Vijay V},
  journal={Combinatorica},
  volume={7},
  number={1},
  pages={105--113},
  year={1987},
  publisher={Springer},
  doi={10.1007/BF02579206}
}

@article{grandoni2014new,
  title={New approaches to multi-objective optimization},
  author={Grandoni, Fabrizio and Ravi, Ramamoorthi and Singh, Mohit and Zenklusen, Rico},
  journal={Mathematical Programming},
  volume={146},
  pages={525--554},
  year={2014},
  publisher={Springer},
  doi={10.1007/s10107-013-0703-7}
}

@article{edmonds1965maximum,
  title={Maximum matching and a polyhedron with 0,1-vertices},
  author={Edmonds, Jack},
  journal={Journal of Research of the National Bureau of Standards B},
  volume={69},
  number={1-2},
  pages={125--130},
  year={1965},
  doi={10.6028/jres.069B.013}
}

@article{stromquist1985sets,
  title={Sets on which several measures agree},
  author={Stromquist, Walter and Woodall, Douglas R},
  journal={Journal of Mathematical Analysis and Applications},
  volume={108},
  number={1},
  pages={241--248},
  year={1985},
  publisher={Elsevier},
  doi={10.1016/0022-247X(85)90021-6}
}

@inproceedings{kabanets2003derandomizing,
  title={Derandomizing polynomial identity tests means proving circuit lower bounds},
  author={Kabanets, Valentine and Impagliazzo, Russell},
  booktitle={Proceedings of the thirty-fifth annual ACM symposium on Theory of computing},
  pages={355--364},
  year={2003},
  doi={10.1145/780542.780595}
}

@article{gurjar2017exact,
  title={Exact perfect matching in complete graphs},
  author={Gurjar, Rohit and Korwar, Arpita and Messner, Jochen and Thierauf, Thomas},
  journal={ACM Transactions on Computation Theory (TOCT)},
  volume={9},
  number={2},
  pages={1--20},
  year={2017},
  publisher={ACM New York, NY, USA},
  doi={10.1145/3041402}
}

@inproceedings{chekuri2011multi,
  title={Multi-budgeted matchings and matroid intersection via dependent rounding},
  author={Chekuri, Chandra and Vondr{\'a}k, Jan and Zenklusen, Rico},
  booktitle={Proceedings of the twenty-second annual ACM-SIAM symposium on Discrete Algorithms},
  pages={1080--1097},
  year={2011},
  organization={SIAM},
  doi={10.1137/1.9781611973082.82}
}

@inproceedings{grandoni2010approximation,
  title={Approximation schemes for multi-budgeted independence systems},
  author={Grandoni, Fabrizio and Zenklusen, Rico},
  booktitle={European Symposium on Algorithms},
  pages={536--548},
  year={2010},
  organization={Springer},
  doi={10.1007/978-3-642-15775-2_46}
}

@article{yuster2012almost,
  title={Almost exact matchings},
  author={Yuster, Raphael},
  journal={Algorithmica},
  volume={63},
  pages={39--50},
  year={2012},
  publisher={Springer},
  doi={10.1007/978-3-540-74208-1_21}
}

@article{cardinal2025inapproximability,
  title={Inapproximability of shortest paths on perfect matching polytopes},
  author={Cardinal, Jean and Steiner, Raphael},
  journal={Mathematical Programming},
  volume={210},
  number={1},
  pages={147--163},
  year={2025},
  publisher={Springer},
  doi={10.1007/s10107-023-02025-4}
}

@article{zhang1995hamiltonian,
  title={Hamiltonian weights and unique 3-edge-colorings of cubic graphs},
  author={Zhang, Cun-Quan},
  journal={Journal of Graph Theory},
  volume={20},
  number={1},
  pages={91--99},
  year={1995},
  publisher={Wiley Online Library},
  doi={10.1002/jgt.3190200110}
}

@article{dershowitz1990cycle,
  title={The cycle lemma and some applications},
  author={Dershowitz, Nachum and Zaks, Shmuel},
  journal={European Journal of Combinatorics},
  volume={11},
  number={1},
  pages={35--40},
  year={1990},
  publisher={Academic Press},
  doi={10.1016/S0195-6698(13)80053-4}
}

@article{zhang2001strong,
  title={On strong 1-factors and Hamilton weights of cubic graphs},
  author={Zhang, Cun-Quan},
  journal={Discrete Mathematics},
  volume={230},
  number={1-3},
  pages={143--148},
  year={2001},
  publisher={Elsevier},
  doi={10.1016/S0012-365X(00)00077-7}
}

@article{edmonds1965paths,
  title={Paths, trees, and flowers},
  author={Edmonds, Jack},
  journal={Canadian Journal of Mathematics},
  volume={17},
  pages={449--467},
  year={1965},
  publisher={Cambridge University Press},
  doi={10.4153/CJM-1965-045-4}
}

@article{camerini1992random,
  title={Random pseudo-polynomial algorithms for exact matroid problems},
  author={Camerini, Paolo M. and Galbiati, Giulia and Maffioli, Francesco},
  journal={Journal of Algorithms},
  volume={13},
  number={2},
  pages={258--273},
  year={1992},
  publisher={Elsevier},
  doi={10.1016/0196-6774(92)90018-8}
}

@article{nagele2024congruency,
  title={Congruency-constrained TU problems beyond the bimodular case},
  author={N{\"a}gele, Martin and Santiago, Richard and Zenklusen, Rico},
  journal={Mathematics of Operations Research},
  volume={49},
  number={3},
  pages={1303--1348},
  year={2024},
  publisher={INFORMS},
  doi={10.1287/moor.2023.1381}
}

@article{nagele2024advances,
  title={Advances on strictly $\Delta$-modular IPs},
  author={N{\"a}gele, Martin and N{\"o}bel, Christian and Santiago, Richard and Zenklusen, Rico},
  journal={Mathematical Programming},
  pages={1--30},
  year={2024},
  publisher={Springer},
  doi={10.1007/s10107-024-02148-2}
}

@inproceedings{aprile2025integer,
  title={Integer programs with nearly totally unimodular matrices: the cographic case},
  author={Aprile, Manuel and Fiorini, Samuel and Joret, Gwena{\"e}l and Kober, Stefan and Seweryn, Mieha{\l} T and Weltge, Stefan and Yuditsky, Yelena},
  booktitle={Proceedings of the 2025 Annual ACM-SIAM Symposium on Discrete Algorithms (SODA)},
  pages={2301--2312},
  year={2025},
  organization={SIAM}, doi={10.1137/1.9781611978322.76}
}

@book{munkres-topology,
  author = {Munkres, James},
  title = {Topology (Second Edition)},
  publisher = {Pearson New International Editions},
  year = {2014},
  isbn = {978-0131816299}
}

@article{PadbergRao,
author = {Padberg, Manfred W. and Rao, M. R.},
title = {Odd minimum cut-sets and $b$-matchings},
journal = {Mathematics of Operations Research},
volume = {7},
number = {1},
pages = {67-80},
year = {1982},
doi = {10.1287/moor.7.1.67},
}

\end{document}